\documentclass[a4paper,10pt,twoside]{amsart}
\usepackage[utf8]{inputenc}
\usepackage{indentfirst}
\usepackage{fancyhdr}
\usepackage{amssymb}
\usepackage{amsmath}
\usepackage{amsthm}
\usepackage{latexsym}
\usepackage{eucal}
\usepackage{pdfpages}
\usepackage{setspace}
\usepackage{braket}
\usepackage{version}
\usepackage{verbatim}
\usepackage{bbm}
\usepackage{subfig}
\usepackage{xcolor}

\usepackage{dsfont}

\usepackage[a4paper,top=3.5cm,bottom=3.5cm,left=3.5cm,right=3.5cm]{geometry}

\usepackage{mathtools}

\DeclareMathAlphabet{\mathcal}{OMS}{cmsy}{m}{n}

\numberwithin{equation}{section}

\theoremstyle{plain}
\newtheorem{theorem}{Theorem}[section]
\newtheorem{lemma}[theorem]{Lemma}
\newtheorem{proposition}[theorem]{Proposition}

\theoremstyle{definition}
\newtheorem{remark}[theorem]{Remark}

\newtheorem*{namedthm*}{\namedthmname}

\renewcommand{\div}{\mathrm{div \ }}

\newcommand{\uNe}{u_{<N/8}}
\newcommand{\huNe}{\hat{u}_{<N/8}}
\newcommand{\fNe}{f_{<N/8}}
\newcommand{\hfNe}{\hat{f}_{<N/8}}
\newcommand{\AfNe}{(Af)_{<N/8}}

\newcommand{\hphi}{\widehat{\phi}}

\newcommand{\hf}{\hat{f}}
\newcommand{\hu}{\hat{u}}

\newcommand{\scalar}[2]{\langle #1, #2 \rangle}
\newcommand{\Scalar}[2]{\big\langle #1, #2 \big\rangle}

\newcommand{\floor}[1]{\lfloor #1 \rfloor}

\newcommand{\dt}{\partial_t}
\newcommand{\Dt}{\frac{d}{dt}}
\newcommand{\dx}{\partial_x}
\newcommand{\dy}{\partial_y}
\newcommand{\dz}{\partial_z}

\newcommand{\dyy}{\partial_{yy}}
\newcommand{\dzz}{\partial_{zz}}

\newcommand{\dtau}{\partial_{\tau}}

\newcommand{\norma}[2]{\lVert #1 \rVert_{#2}}
\newcommand{\Norma}[2]{\big\lVert #1 \big\rVert_{#2}}
\newcommand{\Normabigg}[2]{\bigg\lVert #1 \bigg\rVert_{#2}}

\usepackage{etoolbox}
\patchcmd{\section}{\scshape}{\bfseries}{}{}
\makeatletter
\renewcommand{\@secnumfont}{\bfseries}
\makeatother

\author[M. Dolce]{Michele Dolce}
\title[Nonlinear inviscid damping for zero mean perturbation of the Couette flow]{Nonlinear inviscid damping for zero mean perturbation of the 2D Euler Couette flow}
\address{\textsc{GSSI-Gran Sasso Science Institute}, Viale F. Crispi 7, L'Aquila, 67100, IT}
\email{michele.dolce@gssi.it}

\makeatletter
\let\newauthor\@author
\let\runtitle\@title
\makeatother

\usepackage{authoraftertitle}

\begin{document}
		\maketitle
 \vspace{-0.8cm}
		\begin{abstract}
			In this note we revisit the proof of Bedrossian and Masmoudi \cite{BM} about the inviscid damping of planar shear flows in the 2D Euler equations under the assumption of \textit{zero mean perturbation}. We prove that a small perturbation to the 2D Euler Couette flow in $\mathbb{T}\times \mathbb{R}$ strongly converge to zero, under the additional assumption that the average in $x$ is always zero. In general the mean is not a conserved quantity for the nonlinear dynamics, for this reason this is a particular case. Nevertheless our assumption allow the presence of \textit{echoes} in the problem, which we control by an approximation of the weight built in \cite{BM}. The aim of this note is to present the mathematical techniques used in \cite{BM} and can be useful as a first approach to the nonlinear inviscid damping. 
		\end{abstract}
		\tableofcontents
		\section{Introduction}
		Consider the 2D Euler equation written in vorticity form,
		\begin{align}
		\label{2deuler}&\dt \widetilde{\omega}+v\cdot \nabla \widetilde{\omega}=0, \ \ \ \text{in } \mathbb{T}\times \mathbb{R},
		\end{align}
		where $v$ is the velocity and $\widetilde{\omega}=\operatorname{curl}(v)$ is the vorticity. The Couette flow, given by $v=(y,0)^T$ and $\widetilde{\omega}=-1$, it is a steady state for \eqref{2deuler}. We are interested in perturbations of the form $v=(y,0)^T+(U^x,U^y)^T$ with total vorticity $\widetilde{\omega}=\omega-1$, where $\omega=\operatorname{curl}(U)$. The equations for a perturbation around the Couette flow are given by
		\begin{equation}
		\label{eq:vort}
		\begin{split}
		&\dt \omega+y\dx \omega +U\cdot \nabla \omega=0, \ \ \ \text{in } \mathbb{T}\times \mathbb{R}, \\
		&U=\nabla^{\perp}\psi,\\
		&\Delta \psi= \omega,
		\end{split}
		\end{equation}
		 where $\psi$ is the stream function associated to $U$.\\ \indent
		We provide a proof of the nonlinear \textit{inviscid damping} result obtained by Bedrossian and Masmoudi in \cite{BM}, under the additional assumption that $\int_\mathbb{T} U^x(t)dx=0$ for all times $t$. Let us remark that this condition does not hold true for general solutions to \eqref{eq:vort} and we comment more later on. \\ \indent 
		The purpose of this note is to familiarize with the techniques used in \cite{BM} and to show some of the difficulties of this problem. Event though several simplifications due to our hypothesis the phenomenon of \textit{echoes} is still present, that has analogies with the Landau damping setting, see \cite{bedrossian2016landau,MV}.\\ \indent
		With inviscid damping, roughly speaking, we mean that $U$ will converge to zero even without viscosity, in particular one has something like $\displaystyle \norma{U}{L^2}\leq\epsilon(1+t)^{-1}$, where $\epsilon$ is related to the initial viscosity. For a detailed and clear discussion on the problem we refer to \cite{bedrossianGermainMasmoudi,bedrossian2013asymptotic,BM} and references therein. \\ \indent
		The techniques used here and the main ideas of the proof are the same of \cite{BM}, for this reason we also keep some notation used in the original paper in order to simplify the reading of the latter and related works of Bedrossian, Masmoudi and coauthors see \cite{bedrossian2017stability,bedrossian2016landau,bedrossian2018sobolev,deng2018long}. In particular we will try to highlight when our assumption plays a role by referring always to \cite{BM}.\\
		 Let us also mention the previous work of Lin and Zeng \cite{lin2011inviscid}, which provide an example of instability, and the recent work of Ionescu and Jia \cite{ionescu2018inviscid}, where with the techniques introduced in \cite{BM}, they prove nonlinear inviscid damping in a periodic channel, i.e. $\mathbb{T}\times [-1,1]$, for a perturbation to Couette, in the critical regularity requirement of \cite{BM}, namely $s=1/2$, and controlling the boundary effect on the stream function. They have to assume the initial vorticity to be compactly supported inside the channel, which guarantees that the support of the vorticity will remain inside, as pointed out also in \cite[Remark 5]{BM} for the domain $\mathbb{T}\times \mathbb{R}$. For other type of shear flows, also in the viscous setting, we refer to \cite{bedrossian2017enhanced,zelati2018relation,zelati2019enhanced,lin2004nonlinear,zillinger2016linear,zillinger2017linear}. For perturbations around vortex structures we mention \cite{bedrossian2017vortex,coti2019degenerate,gallay2005global}.\\ \\ \indent
		Let us comment more on our assumption. Since the average in $x$ is not conserved in the nonlinear setting, one has to deal with the zero mode, namely $U_0(t,y)=\int_{\mathbb{T}}U dx$, which gives another shearing component in addition to the one given by the underlying shear flow $(y,0)^T$. In fact thanks to $\div U=0$, one knows that $\int_\mathbb{T} U^ydx =const.$, and so the zero mode of the velocity can be chosen as $U_0(t)=(U_0(t,y),0)^T$.\\ \indent
		The presence of the zero mode implies additional difficulties in the proof of inviscid damping, which are not only mathematical ones. In fact $U_0$ has a relevant influence in the dynamics, in particular the total flow will not converge exactly to the Couette one but it will goes somehow sufficiently 'close', which means that there exist a shear flow $U_\infty=(U_\infty(y),0)$, with (heuristically) $|U_\infty'|\approx 1$, $|U_\infty''|\approx 0$, such that $\norma{U-U_\infty}{L^2}\leq \epsilon(1+t)^{-1}$. The scattering profile is determined by the evolution of the zero mode, for this reason we cannot hope to recover the real scattering in our setting. This is one of the reason why we have to consider our assumption as a sort of toy case. \\ \indent
	    From the mathematical point of view, assuming $U_0=0$ simplifies the proof of the nonlinear inviscid damping. In particular the major simplifications are due to the fact that we can perform estimates in the moving frame generated by following only the background Couette flow. In the general case of $U_0\neq 0$, since the background shear flow depends on the solution itself, the estimates in the moving frame are more involved. In particular a major problem is to translate estimates from the moving frame to the original one. Indeed one needs to ensure that the change of variables is sufficiently smooth. This requires a control also in derivatives of $U_0$. Instead with our assumption, the change of variables is straightforward, hence we have less terms to control.\\ \\ \indent
		As said before, the main ideas and techniques are the same of \cite{BM}. In particular a key point of \cite{BM} is the construction of a proper weight, here we construct the weight in slightly different way, not as sharp as in \cite{BM}, because we want to give an intuitive qualitative behaviour of it. Also we give the proof of its properties in a more heuristic fashion to keep the discussion light, since the real proofs are essentially combinatorial.\\ \indent
		The proof of inviscid damping instead relies on paraproduct decomposition, where essentially we follow the step of \cite{BM}, trying, when possible, to simplify or rephrase some estimates in our setting. \\ \\ \indent
       The note is organized as follows, firstly we introduce some notation and convention that will be used throughout all the note. Then we give also a short discussion on the linear problem to show some of the main difficulties that one may expect.\\ \indent 
       In Section \ref{secstatement} we state the main theorem.\\ \indent
       In Section \ref{secproof}, first we define the main ingredients required and then we reduce the result in the proof of some propositions. Based on those propositions we prove the theorem thanks to bootstrap argument and all the rest of the note is dedicated in proving the propositions. 
       \subsection{Notation and convention}
       With the symbol $a\lesssim b$ we mean that $a\leq Cb$ for some constant $C\geq 1$ that may depend on given quantities. Analogously for $\gtrsim$. We denote $a\approx b$ if $C^{-1}b\leq a \leq Cb$.\\  
       	Define the set of dyadic integers as 
       \begin{equation*}
       \mathbf{D}=\bigg \{\frac{1}{2},1,2,...,2^j,...\bigg\},
       \end{equation*}
       and usually numbers on this set are denoted as $N,N'$. \\ \indent
       For a vector in $\mathbb{R}^{n}$, we denote $|v|=|v_1|+...+|v_n|$, and define the japanese bracket as 
       \begin{equation*}
       \langle v \rangle^2 =1+|v|^2.
       \end{equation*}
       Given two functions $f,g\in L^2(\mathbb{T}\times \mathbb{R})$, we denote the $L^2$ scalar product as $\scalar{f}{g}$.\\
       The Fourier transform of a function $f$ is defined as 
       \begin{equation*}
       \hat{f}_k(\eta)=\frac{1}{2\pi}\int_{\mathbb{T}\times \mathbb{R}}e^{-i(kx+\eta y)}f(x,y)dxdy,
       \end{equation*}
       and the inverse is given by 
       \begin{equation*}
       f(x,y)=\frac{1}{2\pi}\sum_{k\in \mathbb{Z}}\int_\mathbb{R}e^{i(kx+\eta y)}\hat{f}_k(\eta)d\eta.
       \end{equation*}
       We recall also some basic properties, 
       \begin{align*}
       \scalar{f}{g}=\scalar{\hat{f}}{\hat{g}},\\
       \widehat{fg}=\hat{f}*\hat{g},\\
       \big(\widehat{\nabla f}\big)_k(\eta)=(ik,i\eta)\hat{f}_k(\eta).
       \end{align*}
       We define the Gevrey-s norm, with $\sigma$-Sobolev correction, as 
       \begin{equation}
       \norma{f}{\mathcal{G^{\lambda,\sigma}}}^2=\sum_{k\in \mathbb{Z}}\int_\mathbb{R} e^{2\lambda|k,\eta|^s}\langle k,\eta \rangle^{2\sigma}|\hf|^2(k,\eta)d\eta.
       \end{equation}
       The space equipped with this norm is denoted by $\mathcal{G}^{\lambda,\sigma,s}$ and we say that $f\in \mathcal{G}^{\lambda,0,s}$ is a function of Gevrey-$1/s$ class. If $f\in \mathcal{G}^{\lambda,0,1}$, we have an analytic function where $\lambda$ is the radius of analiticity, so for $s<1$ it has an analogous meaning. If $f\in \mathcal{G}^{\lambda,0,0}$ then it is a smooth function.\\
       We will always omit the index $s$ and if $\sigma=0$ we omit also this index. \\ \indent
       Given two Fourier multiplier $B(\nabla)$, $C(\nabla)$, whose symbols are $B_k(\eta), C_k(\eta)$, and two functions $f,g$, we will denote 
       \begin{equation}
       \label{notationBC}
       \mathcal{F}\big(BfCg\big)_k(\eta)=\big(B\hf*C\hat{g}\big)_k(\eta) :=(B\circledast C)(\hat{f}*\hat{g})_k(\eta),
       \end{equation}
       and the notation is introduced somehow in analogy with tensor product of linear maps.\\ 
       When dealing with paraproducts, see \ref{appelitt}, when we have a term like 
       \begin{equation*}
       \big(h_{<N/8}*g_{N}\big)_k(\eta)=\sum_l \int_{\mathbb{R}}h_{k-l}(\eta-\xi)_{<N/8}g_l(\xi)_Nd\xi,
       \end{equation*}
       we always use the convention that $(l,\xi)$ is for high frequencies, namely $N$, and $(k-l,\eta-\xi)$ is for the low one, namely $<N/8$.
       \subsection{Short discussion on the linear problem}
       The linear problem related to \eqref{eq:vort}, is given by the following transport equation 
       \begin{equation}
       \label{eq:linomega}
       \dt \omega+ y\dx \omega=0,
       \end{equation}
       and the velocity $U$ is recovered again thanks to the Biot-Savart law $\eqref{eq:vort}_2$. Clearly the explicit solution is given by following the characteristics, so for example introduce the change of variable $(z,y)\rightarrow (x-yt,y)$ which changes the differential operators as follows 
       \begin{equation}
       \label{DeltaL}
      \begin{split}
       &\nabla_{x,y} \rightarrow (\dz,\dy-t\dz)^T:=\nabla_L,\\
       &\Delta_{x,y}\rightarrow\dzz+(\dy-t\dz)^2:=\Delta_L.
      \end{split}
       \end{equation} 
       So defining $f(t,z,y)=\omega(t,z+ty,y)$ one gets that $\dt f=0$, hence $f(t,z,y)=\omega_{in}(x,y)$. Calling $\tilde{u}(t,z,y)=U(t,z+ty,y)$, we can compute explicitly the velocity on the Fourier side. In fact assuming $\mathbf{k\neq 0}$ (in the linear setting the mode $k=0$ is conserved), there is no problem in defining $U=\nabla^\perp \Delta^{-1}\omega$, and so in the new variables we have
      \begin{align}
      \label{u1} \widehat{\tilde{u}}_{1,k}(t,\eta)=\frac{i(\eta-kt)}{k^2+(\eta-kt)^2}\hat{\omega}_{in}(k,\eta),\\
      \label{u2}\widehat{\tilde{u}}_{2,k}(t,\eta)=-\frac{ ik}{k^2+(\eta-kt)^2}\hat{\omega}_{in}(k,\eta).
      \end{align}
       Then just observe that,
       \begin{equation}
       \label{japbound}
       \langle \eta \rangle^2 \langle\eta-kt\rangle^2 \gtrsim \langle t\rangle^2,
       \end{equation} 
       and thanks to Plancharel and change of variables, one has 
       \begin{equation}
       \label{lininvdamp}
       \begin{split}
       &\norma{U^x}{L^2}\lesssim \frac{1}{\langle t \rangle }\norma{\omega_{in}}{H^1_{x,y}},\\
       &\norma{U^y}{L^2}\lesssim \frac{1}{\langle t \rangle^2 }\norma{\omega_{in}}{H^2_{x,y}},
       \end{split}
       \end{equation}
       and the last two inequalities means \textit{inviscid damping}, i.e. the possibility for the velocity to be damped even without viscosity. Also the equations in \eqref{u1}, \eqref{u2} are the key observation to see a possibility of a \textit{transient growth}, near the \textit{Orr's} critical time $t=\eta/k$, called like this because already Orr in the 1907, see \cite{orr1907stability}, noticed essentially this behaviour. Somehow a related mechanism of transient growths creates a lot of difficulties in the nonlinear case, where frequencies are not independent.\\
       Also in \eqref{lininvdamp} we see that one has to pay regularity to get a decay. \\ \indent
       For more discussion about the linear case, also for shear flows close to Couette and in the channel, we refer to \cite{zillinger2016linear,zillinger2017linear}.
       \section{Statement of the theorem} \label{secstatement}
       As suggested by the linear case, if one tries to perform a simple perturbative argument, any estimate will require to control at least two derivatives more. So due to the structure of the problem, it is natural to ask for an infinite regularity class, in particular Gevrey-$1/s$, where $s>1/2$. By now this restriction is seen to be somehow sharp, because Deng and Masmoudi, see \cite{deng2018long}, have proven that for $s<1/2$ there is instability created completely by a nonlinear mechanism. \\
       In our setting of zero $x$-average, the theorem of Bedrossian and Masmoudi, \cite[Theorem 1]{BM}, reads as follows.
       \begin{theorem}
       	\label{maintheorem}
       	Assume $\int_{\mathbb{T}}  U^xdx=0$ for all times. For all $1/2<s\leq1$, $\lambda_0>\lambda'>0$. Then there exists an $\epsilon_0(\lambda_0,\lambda',s)\leq 1/2$ such that for any $\epsilon<\epsilon_0$, if $\omega_{in}$ satisfies $\int \omega_{in} dx dy=0$, $\int|y\omega_{in}|dxdy<\epsilon$ and 
       	\begin{equation}
       	\norma{\omega_{in}}{\mathcal{G}^{\lambda_0}}\leq \epsilon,
       	\end{equation}
       	then it holds 
       	\begin{align}
       	\label{vorticity}
       	\norma{\omega(t,x+yt,y)}{\mathcal{G^{\lambda'}}}&\leq \epsilon,\\
       	\label{U1}\norma{U^x}{L^2}&\leq \frac{\epsilon}{\langle t \rangle},\\
       	\label{U2}\norma{U^y}{L^2}&\leq \frac{\epsilon}{\langle t \rangle^2}.
       	\end{align}
       \end{theorem}
   For completeness and to make comparison, we state also the theorem for the general dynamics. 
   \begin{theorem}[Bedrossian and Masmoudi \cite{BM}]
   	\label{thbedmas}
   	For all $1/2<s\leq1$, $\lambda_0>\lambda'>0$. Then there exists an $\epsilon_0(\lambda_0,\lambda',s)\leq 1/2$ such that for any $\epsilon<\epsilon_0$, if $\omega_{in}$ satisfies $\int \omega_{in} dx dy=0$, $\int|y\omega_{in}|dxdy<\epsilon$ and
   	\begin{equation*}
   	\norma{\omega_{in}}{\mathcal{G}^{\lambda_0}}\leq \epsilon,
   	\end{equation*}
   	then there exist $f_\infty$ with $\int f_\infty dx dy=0$ and $\norma{f_\infty}{\mathcal{G^{\lambda'}}}\leq \epsilon$ such that 
   	\begin{equation*}
   	\norma{\omega(t,x+yt+\Phi(t,y),y)-f_\infty}{\mathcal{G^{\lambda'}}}\lesssim \frac{\epsilon^2}{\langle t \rangle},
   	\end{equation*}
   	where $\Phi$ is given  explicitly by 
   	\begin{equation*}
   	\Phi(t,y)=\frac{1}{2\pi}\int_0^t \int_{\mathbb{T}}U^x(\tau,x,y)dxd\tau=u_\infty(y)t+O(\epsilon^2),
   	\end{equation*}
   	with $u_\infty=\displaystyle \dy \dyy^{-1} \frac{1}{2\pi}\int_{\mathbb{T}}f_\infty (x,y)dx$. Moreover, the velocity field $U$ satisfies 
   		\begin{align*}
   		\Normabigg{\frac{1}{2\pi}\int_\mathbb{T} U^x(t,x,\cdot)dx-u_\infty}{\mathcal{G^{\lambda'}}}&\lesssim \frac{\epsilon^2}{\langle t \rangle^2},\\
   		\Normabigg{U^x(t)-\frac{1}{2\pi}\int_\mathbb{T} U^x(t,x,\cdot)dx}{L^2}&\lesssim \frac{\epsilon}{\langle t \rangle},\\
   		\Norma{U^y(t)}{L^2}&\lesssim \frac{\epsilon}{\langle t \rangle^2}
   		\end{align*}
   \end{theorem}
   \begin{remark}
   	Clearly our statement is a simpler version of the one in \cite{BM}, because we do not care about $\int_\mathbb{T} U^xdx$, which modify the background shear flow that one has to follow.
   \end{remark}
\begin{remark}
	The theorem in our setting resembles the linear case, in fact our assumption destroys the nonlinear interactions of the zero mode, in particular we lose also information about the scattering of $\omega$ to a steady profile.
\end{remark}
\begin{remark}
	We stress again that the assumption of zero mean it does not have any rigorous justification, and due to the fact that we lose some important information about the general dynamics, theorem \ref{maintheorem} should be considered as a particular case useful to familiarize with the techniques to prove theorem \ref{thbedmas}. 
\end{remark}
\begin{remark}
	As pointed out in \cite{BM}, the spatial localization $\int |y\omega_{in}|dxdy<\epsilon$ is used only to ensure that $U^x\in L^2$. Actually for them it is useful to control the change of coordinate, in our case we do not need such hypothesis, but we have included it just to compare the two theorems.
\end{remark}
For more discussion on implications of the general dynamics we refer to \cite{BM}.
       \section{Proof of Theorem \ref{maintheorem}}
       \label{secproof}
       Thanks to the assumption of zero average, we perform the change of variable as in the linear case, namely $z\rightarrow x-yt$ and define 
       \begin{align*}
       &f(t,z,y)=\omega(t,z+yt,y),\\
       &\tilde{u}(t,z,y)=U(t,z+yt,y),\\
       &\phi(t,z,y)=\psi(t,z+yt,y).
       \end{align*} 
      The system \eqref{eq:vort} becomes
       \begin{equation}
       \label{futilde}
       \begin{split}
       &\dt f+ \tilde{u} \cdot\nabla_{L} f=0,\\
       &\tilde{u}=\nabla^\perp_L\phi=\nabla^\perp_L\Delta_L^{-1}f,
       \end{split}
       \end{equation} 
       where $\nabla_{L}, \Delta_L$ are defined in \eqref{DeltaL}.\\ 
      Now observe that in \eqref{futilde} there is a crucial cancellation that eliminates the time dependent factor of $\nabla_{L}$, in fact \eqref{futilde} is exactly the same as 
      \begin{equation}
      \label{fu}
      \begin{split}
      &\dt f+u\cdot \nabla_{z,y}f=0,\\
      &u:=\nabla^\perp_{z,y} \phi= \nabla^\perp_{z,y}\Delta_L^{-1}f,
      \end{split}
      \end{equation}
      which is the one that we want to investigate. A similar cancellation holds also in the general case, where one has to chose carefully the change of coordinates. In contrast to \cite{BM}, here $u$ is again divergence free. From now on we will omit the dependence of the gradients in $(z,y)$, since it will be clear when we are working on the new reference frame. \\ \indent
      The main hint of the linear case, is that if one wants decay on the velocity, then it is sufficient to bound $f$ in a suitable high regularity space, because if we can pay at least an $H^1$ price on $f$, then the decay is naturally given by the negative order differential operator in $\eqref{futilde}_2$. The nonlinear structure of the problem tells that, even arguing heuristically, a bound in $H^s$ implies at least a control on $H^{s+2}$. This is the main reason why it is necessary an infinite regularity setting. \\ \indent
      For this purpose, we introduce the multiplier 
      \begin{equation}
      \label{def:A}
      A_k(t,\eta)=e^{\lambda(t)|k,\eta|^s}\langle k,\eta \rangle^\sigma J_k(t,\eta),
      \end{equation}
      where $\lambda(t)$, chosen by the proof, has to satisfy the following ODE:
      \begin{equation}
      \label{deflambda}
      \begin{split}
      &\dot{\lambda}(t)=-\frac{\delta_\lambda}{\langle t \rangle^{2s}}\big(1+\lambda(t)\big) \ \ t\geq 1\\
      &\lambda(t)=\frac{3}{4}\lambda_0+\frac{1}{4}\lambda'\ \ t\leq 1,
      \end{split} 
      \end{equation}
      where $\delta_\lambda \approx \lambda_0-\lambda'$ is a small parameter to ensure that $\lambda(t)>\lambda_0/2+\lambda'/2$.\\ \indent
       The Sobolev regularity, with $\sigma>10$, is considered just to avoid technicalities to close estimates with the same index of Gevrey regularity. Instead the most important part is the multiplier $J$, defined as 
      \begin{equation}
      \label{def:J}
      J_k(t,\eta)=\frac{e^{\mu |\eta|^{1/2}}}{w_k(t,\eta)}+e^{\mu|k|^{1/2}},
      \end{equation}
      and the key element is the weight $w$, to be constructed later based on a toy model of the 'worst possible case'. In particular, $1/w$ will grow as $e^{\mu/2 \sqrt{\eta}}$, and so we control growth in time by paying this amount of derivatives. 
      \begin{remark}
      	The introduction of this weight maybe is one of the key point in the proof of \cite{BM} and we discuss more about it in Section \ref{secweight}.
      \end{remark}
      For convenience define also 
      \begin{align}
      \label{def:Jtilde}
      \tilde{J}_k(t,\eta)&=\frac{e^{\mu |\eta|^{1/2}}}{w_k(t,\eta)}\\
      \tilde{A}_k(t,\eta)&=e^{\lambda(t)|k,\eta|^s}\langle k,\eta \rangle^\sigma \tilde{J}_k(t,\eta).
      \end{align} 
      Just notice that it holds $\tilde{A}\leq A$ and if $|k|\leq |\eta|$, then $A\lesssim \tilde{A}$.\\ \indent
      	As pointed out also in \cite{BM}, to treat the analytic setting, namely $s=1$, one should add an additional Gevrey correction, since if $s<1$, we have an improved triangular inequality, see lemma \ref{lemmaconc}, that helps for example in keeping always the same radius of regularity $\lambda$. For instance it is enough to include an additional Gevrey-$1/\beta$ regularity, with $1/2<\beta<1$.\\ \indent

      Finally the proof of the theorem is reduced to prove that the following energy is bounded,
      \begin{equation}
      E(t)=\frac{1}{2}\norma{A(t)f(t)}{L^2_{z,y}}^2,
      \end{equation}
      because in fact we have that 
      \begin{equation}
      \label{omegalessAf}
      \norma{\omega(t,x+yt,y)}{\mathcal{G}^{\lambda'}}\leq\norma{Af}{L^2}.
      \end{equation}
      The idea is to proceed via a bootstrap argument, and to start the argument there is local well-posedness for 2D-Euler in Gevrey spaces. For this and related results see \cite{bardos1977domaine,ferrari1998gevrey,foias1989gevrey,kukavica2009radius,levermore1997analyticity}. Let us recall what we need.
      \begin{lemma}
      	\label{localwell}
      	For all $\epsilon>0$, $s>1/2$ and $\lambda_0>\lambda'>0$, there exists an $\epsilon'>0$ such that if $\norma{\omega_{in}}{\mathcal{G^{\lambda_0}}}\leq \epsilon'$ and $\int|y\omega_{in}|dxdy\leq\epsilon'$, then 
      	\begin{align*}
      	&\sup_{t\in[0,1]}\norma{f(t)}{\mathcal{G}^{3\lambda_0/4+\lambda'/4,\sigma}}<\epsilon,\\
      	&E(1)<\epsilon^2.
      	\end{align*}
      \end{lemma}
      So we safely ignore the interval $[0,1]$. To get a bound for $E(t)$, it is natural to compute the derivative with respect to time to get 
      \begin{equation}
      \label{dtE}
      \begin{split}
      \Dt E(t)&=\scalar{\dt(A)f}{Af}+\scalar{A\dt f}{Af}\\
      & =\dot{\lambda}\scalar{|\nabla|^sAf}{Af}-\scalar{\frac{\dt w}{w}\tilde{A}f}{Af}-\scalar{A(u\cdot \nabla f)}{Af}\\
      &=-CK_\lambda-CK_w-\scalar{A(u\cdot \nabla f)}{Af}.
      \end{split}
      \end{equation}
      where the scalar product is in $(z,y)$, hence $(k,\eta)$ by Plancharel. The terms $CK_\cdot$ stands for Cauchy-Kovalevskaya. The involved Fourier multipliers are positive, hence the associated operators are self-adjoint. So rewrite the $CK$'s terms as 
      \begin{align}
      \label{CKl}
      CK_\lambda&=-\dot{\lambda}\norma{|\nabla|^{s/2}Af}{L^2}^2\\
      \label{CKw}
      CK_{w}&=\bigg \langle\sqrt{\frac{\dt w}{w}}\tilde{A}f,\sqrt{\frac{\dt w}{w}}Af\bigg \rangle\geq \bigg \lVert\sqrt{\frac{\dt w}{w}}\tilde{A} f\bigg \rVert_{L^2}^2.
      \end{align} 
      As said in \cite{BM}, the previous computations are not really rigorous, because it is not ensured that it is possible to take derivatives of $Af$. One can overcome the problem by regularization and passage to the limit. \\
      Now we are ready to make our bootstrap hypothesis for $t\geq1$, namely 
      \begin{align}
      \label{B1}\tag{B1}
      &E(t)\leq 4\epsilon^2\\
      \label{B2}\tag{B2}&\int_1^t(CK_\lambda(\tau)+CK_w(\tau))d\tau\leq 8\epsilon^2.
      \end{align}
      Let us state the main proposition to prove theorem \ref{maintheorem}.
      \begin{proposition}
      	\label{mainprop}
      	Assume that \eqref{B1}-\eqref{B2} holds in a time interval $[1,T^*]$. Then there exists and $0<\epsilon_0<1/2$, depending on $\lambda,\lambda',s,\sigma$, such that if $\epsilon<\epsilon_0$, then for any $t\in [1,T^*]$ it holds
      	\begin{align}
      	&E(t)\leq 2\epsilon^2\\
      &\int_1^t(CK_\lambda(\tau)+CK_w(\tau))d\tau\leq 6\epsilon^2.
      	\end{align}
      \end{proposition}
  \begin{remark}
  	To continue the comparison with \cite{BM}, here the bootstrap hypothesis and so proposition \ref{mainprop}, are simplified. One simplification comes from the fact that we do not have to control our change of variable, being always well defined. 
  \end{remark}
  This proposition means that $T^*=+\infty$, hence proving our theorem. To see that, one can argue as follows: since $E(t)$ is continuous then the set of times on which it holds the property $E(t)\lesssim \epsilon^2$, namely $[1,T^*]$, is closed and open, but since it is also connected it is only possible if $T^*=+\infty$. Otherwise one can think that since $E(T^*)\leq 2\epsilon^2$, then, by the local well-posedness lemma \ref{localwell}, we can argue analogously for the interval $[T^*,T^{**}]$. So one extend $T^*$ up to infinity.\\ \\ \indent
      Since everything is reduced in proving proposition \ref{mainprop}, let us sketch the idea. The last term in \eqref{dtE} is the one that we have to control, here we have the major simplifications with respect to the true case of non zero mean. A first simplification (not really significant) comes from the fact that $\div u=0$, hence one has 
      \begin{equation}
      0=\frac{1}{2}\int\div(u|Af|^2)=\int Af(u\cdot \nabla Af),
      \end{equation}
      and so just by adding zero to the last term in \eqref{dtE}, we obtain a commutator that helps in controlling some term, namely 
      \begin{equation}
      \scalar{Af}{A(u\cdot \nabla f)-u\cdot \nabla Af}=\sum_{N\geq 8}T_N+\sum_{N\geq 8}R_N+\sum_{N\in\mathbf{D}}\sum_{N/8\leq N'\leq 8N}\mathcal{R}_{N,N'},
      \end{equation}
      where we have just performed a paraproduct decomposition, see \ref{appelitt}. Each term is defined as follow 
      \begin{align*}
      T_N&=\scalar{Af}{A(\uNe \cdot \nabla f_N)-\uNe\cdot \nabla(Af)_N},\\
      R_N&=\scalar{Af}{A(u_N \cdot \nabla \fNe)-u_N\cdot \nabla\AfNe},\\
      \mathcal{R}_{N,N'}&=\scalar{Af}{A(u_N \cdot \nabla f_{N'})-u_N\cdot \nabla Af_{N'}}.
      \end{align*}
      
       $T_N$ stand for \textit{transport}, since the 'velocity' $u$ is cut at low frequencies, and so it will be easy to obtain bounds without paying to much regularity. Essentially it is like treating this term as if the $f$ is transported by a passive velocity. Here the assumption of zero $x$-average has not a big influence.\\ \indent
      Instead the most problematic term is $R_N$, called \textit{reaction}, where one has to be really careful to control the possibility of losing regularity in time. In fact the weight $w$ is built up to simulate the behaviour of a particular term of it, namely the one that produces \textit{echoes}. Here it will be crucial to split carefully in different time intervals, to recover integrability in time when possible and otherwise absorbing terms thanks to the weight. \\
      For this term the assumption of zero mean plays a major role to reduce the number of terms that one has to control.  \\ \indent
      Finally $\mathcal{R}_{N,N'}$ is a \textit{remainder} and it is the easiest one to bound.
      \begin{remark}
      	 In the case of non-zero mean, $u$ is not only given by some derivative of the stream function, but it contains also derivatives of the zero mode. For this reason, when $u$ is at high frequency, namely in the reaction term, one should be really careful to obtain precise estimates. In fact, for example, to recover integrability in time, the elliptic estimates has to be obtained following the background shear flow, for us it is trivial, in \cite{BM} it requires a precise elliptic control which has to be carefully obtained, since as already said, the background shear flow changes with the solution itself. 
      \end{remark}
      So let us just summarize in the following propositions the estimates that we have to prove.
      \begin{proposition}[Transport]
      	\label{proptransport}
      	Under the bootstrap hypothesis 
      	\begin{equation*}
  	\sum_{N\geq 8}|T_N|\lesssim \epsilon CK_\lambda+\epsilon CK_w +\frac{\epsilon^3}{\langle t \rangle^2}.
      	\end{equation*}
      \end{proposition}
  \begin{proposition}[Reaction]
  	\label{propreaction}
  	Under the bootstrap hypothesis 
  	\begin{equation*}
  	\sum_{N\geq 8}|R_N|\lesssim \epsilon CK_\lambda+\epsilon CK_w +\frac{\epsilon^3}{\langle t \rangle^2}.
  	\end{equation*}
  \end{proposition}
\begin{proposition}[Remainders]
	\label{propremainders}
	Under the bootstrap hypothesis 
	\begin{equation*}
	\sum_{N\in\mathbf{D}}\sum_{N/8\leq N'\leq 8N}|\mathcal{R}_{N,N'}|\lesssim\frac{\epsilon^3}{\langle t \rangle^2}.
	\end{equation*}
\end{proposition}
If we are able to prove those propositions, thanks to \eqref{dtE} we get that 
\begin{equation*}
E(t)+(1-\epsilon)\int_{1}^{t}(CK_\lambda(\tau)+CK_w(\tau))d\tau \lesssim E(1)+\epsilon^3,
\end{equation*}
hence thanks to bootstrap hypothesis, we get that, for $\epsilon$ sufficiently small, proposition \ref{mainprop} holds. \\ \indent
To conclude, the estimates on the velocity follows exactly as in the linear case, in fact one has 
\begin{align*}
\norma{U^x}{L^2}&=\norma{\dy \Delta^{-1}\omega}{L^2}=\norma{(\dy-t\dz)\Delta_L^{-1}f}{L^2}\\
&\lesssim \frac{1}{\langle t \rangle}\norma{f}{H^1} \lesssim \frac{1}{\langle t \rangle}\norma{Af}{L^2}\lesssim \frac{\epsilon}{\langle t \rangle},
\end{align*}
analogously for $U^y$, hence proving all the estimates in theorem \ref{maintheorem}.
\begin{remark}
	In our setting we pass from the standard reference frame to the one which follows the Couette flow in a trivial way. In the general dynamics, in order to follow the background shear flow, where one is able to obtain estimates, the change of variable to perform it is not trivial at all. Also one should check that it is possible to perform the change of variable in a sufficiently smooth way, because in Gevrey spaces is a delicate task. For this reason in \cite{BM} they have to be really precise in keeping under control the background shear flow.
\end{remark}
      \section{Construction of the weight}
      \label{secweight}
      \subsection*{Heuristic Ideas}
      Consider the equations \eqref{fu}. We want to construct a weight that mimic the 'worst possible case'. In particular, we know that we have to pay regularity on $u$ to have decay, so we are interested to see the interactions between $u$ at high frequencies and $f$ at low ones, since we have done the splitting with the paraproduct. So we are considering 
      \begin{equation*}
      \dt f= -\nabla^\perp (\Delta_L^{-1}f)\cdot  \nabla f_{lo}\approx \dy (\Delta_L^{-1}f)\dz f_{lo}, 
      \end{equation*}
      where we have approximated with the most dangerous case, namely derivatives in $y$. So on the Fourier side we are left with 
      \begin{equation}
      \label{approxw}
     \begin{split}
      \dt \hf(t,k,\eta)&=-\sum_{l\neq 0}\int \frac{\xi (k-l)}{l^2+|\xi-lt|^2}\hf (l,\xi)\hf_{lo}(t,k-l,\eta-\xi)d\xi\\
      &\approx -\sum_{l\neq 0}\frac{\eta (k-l)}{l^2+|\eta-lt|^2}\hf (l,\eta)\hf_{lo}(t,k-l,0)
     \end{split}
      \end{equation}
      since being $\hf_{lo}$ in low frequencies, we have approximated $\eta=\xi$.\\
       For the same reason we are interested in the case $l=k+1$ or $l=k-1$, for example. Which means to see how the $k$-th mode is influenced by nearby frequencies. The mechanism from low to high frequencies it is somehow a natural part, where one just pays some standard regularity to control that, for example the travelling on high frequencies is evident even in the linear case. \\ \indent
       Instead the problem is to control a \textit{high-to-low cascade}, which creates the so called phenomenon of echoes, observed numerically \cite{vanneste2002nonlinear,vanneste1998strong} and experimentally \cite{yu2002diocotron,yu2005phase} in the context of 2D Euler. Mathematically has some relation with \textit{plasma echoes} present in Landau damping, see \cite{bedrossian2016landau,MV}.\\ \indent
       We consider the mode $k$, that at the \textit{resonant time} $t=\eta/k$, has a strong effect on the $(k-1)$-th mode if $\eta/k^2>1$, otherwise it is controlled. Then the mode $(k-1)$ influence the $(k-2)$ and so on, creating this high to low cascade. The simplest model of this effect is built in the following way. Assume that $\hf_{lo}=O(\beta)$ and absolute values everywhere on the previous equations, the toy model under consideration is the following 
       \begin{equation}
       \label{toyfRfNR}
       \begin{split}
       \dt f_R &=\beta \frac{k^2}{|\eta|} f_{NR}\\
       \dt f_{NR}&= \beta \frac{|\eta|}{k^2+|\eta-kt|^2}f_R
       \end{split}
       \end{equation}
       where $f_R$ is the \textit{Resonant} mode, i.e. $k$, and $f_{NR}$ the \textit{Non-Resonant}, i.e. $k-1$. The factor $k^2/|\eta|$ in the ODE for $f_R$ is a rough upper bound of the strongest interaction that a non resonant mode has on a resonant one, for times nearby $t=\eta/k$. In fact for example, consider \eqref{approxw} with $l=k-1$ at time $t=\eta/k$, to get
       \begin{equation*}
       \dt \hf_k\approx \beta \frac{k^2|\eta|}{k^2(k-1)^2+\eta^2}\hf_{k-1}\leq \beta \frac{k^2}{|\eta|}\hf_{k-1}.
       \end{equation*}
       The weight will be constructed by estimating the behaviour of \eqref{toyfRfNR} near times $\eta/k$. \\ \indent
       First of all let us construct time intervals centred in the resonant time $\eta/k$. Since we are interested in $\eta/k^2>1$, this means that $k<\sqrt{\eta}$, hence $t>\sqrt{\eta}$. Then all the resonant times are in the interval $[\sqrt{\eta},\eta]$, so we want to divide it with subintervals centred in $\eta/k$ for $k=1,...,\floor{\sqrt{\eta}} $, where $\floor{\cdot}$ is the integer part. In general one has intervals of the form $[\eta/k-a_k,\eta/k+b_k]$. One possibility is to choose $a_k=b_{k+1}$, i.e the mid point between $\eta/(k+1)$ and $\eta/k$. In particular $a_k$ is given by
       \begin{equation*}
       \frac{\eta}{k}-a_k=\frac{\eta}{k+1}+a_k\Rightarrow a_k=\frac{\eta}{2k(k+1)}.
       \end{equation*}
       So define $t_{|k|,\eta}=|\frac{\eta}{k}|-a_{|k|}$ and the critical intervals as
       \begin{equation}
       \label{defIketa}
       I_{k,\eta}=\begin{cases}
       [t_{|k|,\eta},t_{|k|-1,\eta}] \ \text{if $\eta k \geq 0$ and $1\leq k\leq \floor{\sqrt{\eta}}$}, \\ \emptyset \ \text{otherwise}.
       \end{cases}
       \end{equation}
       The restrictions on $\eta k \geq 0$, are due simply to the fact that for positive times are the only frequencies that can be resonant. This are the intervals considered in \cite{BM}\\
       For convenience we define also 
       \begin{equation}
       \label{defItildeketa}
       \tilde{I}_{k,\eta}=\begin{cases}
       \displaystyle\bigg[\frac{\eta}{k}-\frac{\eta}{k^2},\frac{\eta}{k}+\frac{\eta}{k^2}\bigg] \ \text{if $\eta k \geq 0$ and $1\leq k\leq \floor{\sqrt{\eta}}$}, \\ \emptyset \ \text{otherwise},
       \end{cases}
       \end{equation}
       because we will use those intervals to construct the weight. So let us recall \cite[Proposition 3.1]{BM}, which give an estimate on \eqref{toyfRfNR}. We will also give a short proof.
       \begin{proposition}
       	\label{propGrowth}
       	Let $\tau=t-\eta/k$ and consider $(f_R(\tau),f_{NR}(\tau))$ solution of \eqref{toyfRfNR} with $f_R(-\eta/k^2)=f_{NR}(-\eta/k^2)=1$.\\ 
       	Then $f_{R}\approx\displaystyle  \frac{k^2}{|\eta|}(1+|\tau|)f_{NR}$. \\More precisely, there exist a constant $C$ such that for all $\beta <1/2$ and $\eta/k^2\geq 1$, 
       	\begin{align*}
       	f_R(\tau)&\leq C\bigg(\frac{k^2}{\eta}(1+|\tau|)\bigg)^{-C\beta}, \ \ -\frac{\eta}{k^2}\leq \tau \leq 0,\\
       	f_{NR}(\tau)&\leq C\bigg(\frac{k^2}{\eta}(1+|\tau|)\bigg)^{-C\beta-1}, \ \ -\frac{\eta}{k^2}\leq \tau \leq 0,\\
       	       	f_{R}(\tau)&\leq C\bigg(\frac{\eta}{k^2}\bigg)^{C\beta}(1+|\tau|)^{C\beta+1}, \ \ 0\leq\tau\leq\frac{\eta}{k^2},\\
       	       	 	f_{NR}(\tau)&\leq C\bigg(\frac{\eta}{k^2}\bigg)^{C\beta+1}(1+|\tau|)^{C\beta}, \ \ 0\leq\tau\leq\frac{\eta}{k^2},      	
       	\end{align*}
       	where it is fixed $3/2<1+2C\beta<10$.
       \end{proposition}
   \begin{remark}
   	\label{remarkgrowth}
   	Overall in the interval $[-\eta/k^2,\eta/k^2]$, both $f_R$ and $f_{NR}$, are mostly amplified by a factor $C(\frac{\eta}{k^2})^{1+2C\beta}$.
   \end{remark}
   \begin{proof}
   	Let us rewrite the system \eqref{toyfRfNR} in the variable $\tau$, call $\gamma=\displaystyle \frac{|\eta|}{k^2}$, to get
   	\begin{equation}
   	\label{proofprop}
   	\begin{split}
   	\dtau f_R &=\frac{\beta}{\gamma} f_{NR},\\
   	\dtau f_{NR}&= \frac{\beta\gamma}{1+\tau^2}f_R.
   	\end{split}
   	\end{equation}
   	First of all notice that both $f_R$, $f_{NR}$ are increasing, so they will remain positive. Then consider $\tau <0$. Multiply $\eqref{proofprop}_1$ by $(\gamma/\beta)  f_R$, $\eqref{proofprop}_2$ by $\displaystyle \frac{(1+\tau^2)}{\beta\gamma}f_{NR}$ and subtract, to have 
   	\begin{equation*}
   	\frac{1}{2}\frac{d}{d\tau}\bigg(\frac{\gamma}{\beta} f^2_R- \frac{1+\tau^2}{\beta\gamma}f^2_{NR}\bigg)=-\frac{\tau}{\beta\gamma}f^2_{NR}\geq 0, 
   	\end{equation*}
   	since $\tau<0$. This implies
   	\begin{equation}
   	\label{proofprofN}
   	f_{NR}\lesssim \frac{C\gamma}{1+|\tau|}f_{R},
   	\end{equation}
   for some $C$. Plugging the previous estimate on $\eqref{proofprop}_2$, one gets 
   	\begin{equation*}
   	\dtau f_{R}\lesssim \frac{C\beta}{1-\tau}f_{R},
   	\end{equation*}
   	since $\tau \leq0$. Integrating in the interval $[-\eta/k^2,0]$, we have   	
   	\begin{equation}
   	f_R \lesssim \bigg(\frac{k^2}{\eta}(1+|\tau|)\bigg)^{-C\beta} \label{proofpropfR}.
   	\end{equation}
   	This proves the first inequality of the Proposition, plugging \eqref{proofpropfR} into \eqref{proofprofN}, recalling $\gamma= \eta/k^2$, we have also the second one.\\
   	The inequalities for positive times follows by the same argument, just notice that $f_{R}(0)\lesssim (\eta/k^2)^{C\beta}$ and analogously for $f_{NR}(0)$.
   \end{proof}
\subsection*{Maximal growth}
Now let us restate the argument in \cite[Lemma 3.1]{BM} to prove the maximal growth of the weight.\\ 
As pointed in remark \ref{remarkgrowth}, the maximal growth over an interval like $I_{k,\eta}$, defined in \eqref{defIketa}, is something like $C(\frac{\eta}{k^2})^{1+2C\beta}$. This process is significant over the whole interval $[\sqrt{\eta},\eta]$, in which one can accumulate such growth for every $k=1,...,\floor{\sqrt{\eta}}$. Calling $c=1+2C\beta$ and $N=\floor{\sqrt{\eta}}$, this means that one can accumulate the following maximal growth 
\begin{equation*}
M_G=\bigg(\frac{\eta}{N^2}\bigg)^c\bigg(\frac{\eta}{(N-1)^2}\bigg)^c\dots \frac{\eta}{1}^c=\bigg(\frac{\sqrt{\eta}^N}{N!}\bigg)^{2c}.
\end{equation*}
Then, if $\eta\gg1$, thanks to Stirling's formula, $N!\approx \sqrt{2\pi N}N^Ne^{-N}$, we infer 
\begin{equation}
\label{MG}
M_G\approx\bigg(\frac{\sqrt{\eta}^N}{N^N}\bigg)^{2c}\frac{e^{2cN}}{\eta^c}\lesssim \frac{1}{\eta^{\mu/8}}e^{\frac{\mu}{2}|\eta|^{1/2}},
\end{equation}
since $N=\floor{\sqrt{\eta}}$, and defining $\mu=4c$. This is the main fact which explain why one has to take $s\geq1/2$, because it cannot be excluded a priori such a growth based on the toy model. 
\subsection{Construction of $w$}
In \cite{BM}, they construct the weight $w$ starting at time $2\eta$, with $w=1$, and then going backward in time up to $\sqrt{\eta}$, accumulating the \textit{inverse} of the predicted maximal growth \eqref{MG}.\\ \indent
Instead here we construct the weight forward in time, in a less rigorous and precise way. For the real construction see \cite[Section 3]{BM}.\\ \\ \indent
The basic idea is that by the expected maximal growth \eqref{MG}, we need to fit a growth of $e^{\mu\sqrt{\eta}}$ in the interval $[\sqrt{\eta},\eta]$. Notice that we define with $\mu$ and not $\mu/2$ in order to absorb some Sobolev contribution. Then fix $k,\eta$, and let us start constructing $w_k(t,\eta)$. To simplify the notation, think as $\sqrt{\eta}$ to be an integer, in remark \ref{remcorr} we explain how it works in general.\\
By the previous argument, it is reasonable to start with
\begin{equation}
\label{wfort<seta}
w_k(t,\eta)=e^{-\mu\sqrt{\eta}} \ \ \text{for $t<\sqrt{\eta}$},
\end{equation}
 Then we will approximate the non resonant modes by an exponential, namely we want to start at time $t=\sqrt{\eta}$ with the value predicted in \eqref{wfort<seta}, and arrive at time $t=2\eta$ at $1$. Then just by a quadratic connection between those values, we get
\begin{equation}
\label{defwNR}
w_{NR}(t,\eta)=\exp\bigg[{-\mu\bigg(\sqrt{\eta}-\frac{ \sqrt{\eta}}{(2\eta-\sqrt{\eta})^2}(t-\sqrt{\eta})^2}\bigg)\bigg].
\end{equation}
	Clearly the true construction is based on the bounds provided in \ref{propGrowth}, which are better than exponential ones. But in the end, in proving properties, one can obtain at most exponential bounds, as seen in \eqref{MG}. \\
	Since the multiplier for the solution contains $1/w$, with respect to \cite{BM} we are assigning more regularity for more time.
\begin{remark}
	\label{remarkdtw}
	We set up a quadratic connection since we want $\dt w$ to be continuous near $t=\sqrt{\eta}$, which before was $0$, hence it is the simplest choice to retain this property. 
\end{remark}
Finally, since by the proposition \ref{propGrowth}, we know that for $|\tau|<\frac{\eta}{k^2}$ it holds $f_R\approx \displaystyle \frac{k^2}{|\eta|}(1+|\tau|)f_{NR}$. Recalling the definition of $\tilde{I}_{k,\eta}$ given in \eqref{defItildeketa}, we define 
\begin{equation}
\label{defwR}
\begin{split}
w_R(t,\eta)=\frac{k^2}{\eta}\bigg(1+a_{k,\eta}\bigg|t-\frac{\eta}{k}\bigg|\bigg)w_{NR}(t,\eta) \ \ \text{for} \ t \in \tilde{I}_{k,\eta}
\end{split}
\end{equation}
where $a_{k,\eta}$ is chosen such that $\frac{k^2}{\eta}(1+a_{k,\eta}\frac{\eta}{k^2})=1$. In this way $w_R$ and $w_{NR}$ are exactly the same in the extremes of the interval $\tilde{I}_{k,\eta}$. Notice that at time $t=\frac{\eta}{k}$, $w_R(\frac{\eta}{k},\eta)=\frac{k^2}{\eta}w_{NR}(\frac{\eta}{k},\eta)$. \\ 
Then we define the weight as 
\begin{equation}
\label{defw}
w_k(t,\eta)=\begin{cases}
e^{-\mu\sqrt{\eta}} \ \ &t<\sqrt{\eta},\\
w_{NR}(t,\eta) \ \ &t\in [\sqrt{\eta},2\eta] \setminus \tilde{I}_{k,\eta},\\
w_R(t,\eta) \ \ &t\in \tilde{I}_{k,\eta}\\
1 \ \ &t>2\eta.
\end{cases}
\end{equation}
\begin{remark}
	\label{remcorr}
	To construct more precisely the weight, one has to substitute $\sqrt{\eta}$ with $t_{\floor{\sqrt{\eta}},\eta}$. Also, in defining $w_R$, see \eqref{defwR}, it would be more precise to consider the $I_{k,\eta}$, but them are not symmetric as $\tilde{I}_{k,\eta}$, so one should split the interval in two pieces and chose two constant to fit the values at the boundaries of the interval. \\
	We avoided to be really precise, since the construction is already a sort of approximation.
\end{remark}
\begin{remark}
	The weight of Bedrossian and Masmoudi is much sharper than our weight. In particular we are imposing the same exponential scale over the whole interval, so assigning more regularity with respect to the weight of \cite{BM}. This fact reflect also in the resonant interval, where in our case we see a peak, since it grows with the same scale. In the case of \cite{BM} there is not an evident peak since in a single interval the behaviour is polynomial. But notice that both $w_k(t,\eta)^{-1}$ are exactly increased by the same quantity $\eta/k^2$ for the resonant time.
\end{remark}
\begin{figure}[]	
	  \centering
	\hspace*{-3.2cm} 
	\includegraphics[scale=.5]{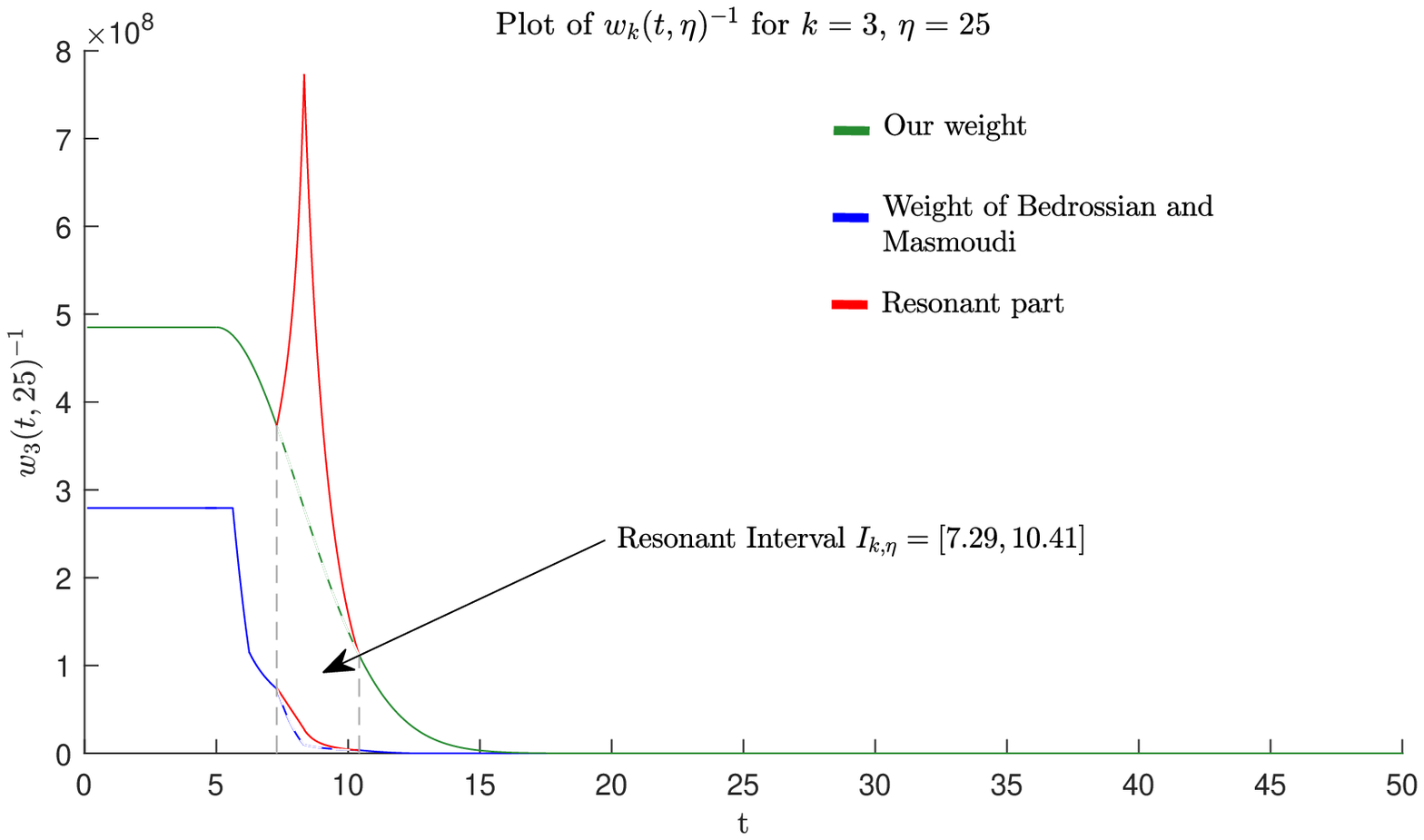}
	
	\caption{It is shown the plot of the weight we have constructed compared with the weight of Bedrossian and Masmoudi built in \cite{BM}. To visualize things we have chosen $c=1.2$ and $\mu=4$ since for the true $c>3/2$ and $\mu=4c$, things are more difficult to clearly visualize in a single plot, but the behaviour is essentially the one of the picture. For our weight we have used the intervals $I_{k,\eta}$ making the corrections stated in remark \ref{remcorr}.  It is interesting to compare the figure with the one of \textit{plasma echoes} that appears in the works about Landau damping of Mohout and Villani \cite[Section 6.3, pg 109]{MV}. For the plot of the weight of \cite{BM} the author thanks Mattia Manucci for the help.}
	\label{Figurew}
\end{figure}
\newpage
\subsection{Properties of $w$}
This section is the most important, since it allows to deal with our multiplier. Here we recall some properties of the weight given in \cite[Section 3]{BM}. We provide some basic ideas of the proof, sometimes based on our toy model weight, the real proof are essentially combinatorial.
\begin{lemma}
	\label{trichotomy}
	Let $\xi, \eta$ such that, for some $\alpha \geq 1$ it holds $\frac{1}{\alpha}|\xi|\leq |\eta|\leq \alpha |\xi|$. Let $k,l$ such that $t\in I_{k,\eta}$ and $t\in I_{l,\xi}$ (note that $l\approx k$). Then at least one of the following holds:
	\begin{itemize}
		\item[(a)] $k=l$ (almost same interval)
		\item[(b)] $|t-\frac{\eta}{k}|\geq \frac{1}{10\alpha}\frac{|\eta|}{k^2}$ and $|t-\frac{\xi}{l}|\geq \frac{1}{10\alpha}\frac{|\xi|}{l^2}$ (far from resonance)
		\item[(c)] $|\frac{\eta}{l}-\frac{\xi}{l}|\gtrsim_\alpha \frac{\eta}{l^2}$ (well-separated)
	\end{itemize}
\end{lemma}
\begin{proof}
	Assume $(a), (b)$ false then $(c)$ is proven by $|x-y|\geq||x|-|y||$ and the definition of the interval $I_{k,\eta}$.
\end{proof}
\begin{lemma}
	\label{lemma:dtw/wapprox}
	For $t\in I_{k,\eta}$ and $t>2\sqrt{\eta}$, let $\tau=t-\frac{\eta}{k}$, it holds 
	\begin{equation}
	\label{dtwR/wR}
	\frac{\dt w_{NR}(t,\eta)}{w_{NR}(t,\eta)}\approx\frac{1}{1+|\tau|}\approx \frac{\dt w_R(t,\eta)}{w_R(t,\eta)}
	\end{equation}
\end{lemma}
\begin{proof}
	The restriction on $t>2\sqrt{\eta}$ is because before that time, essentially $\dt w_{NR}\approx0$. Since it holds also that $\dt w=0$ for $t\geq 2\eta$, then \eqref{dtwR/wR} essentially follows from definition. Otherwise one can think to approximate the exponential with its Taylor series cut at some point, and then the bound just tells us that the derivative of a polynomial loses one degree.  
\end{proof}
In the following lemma we see how it is possible to exchange weight at different frequencies.
\begin{lemma}
	\label{lemmadtw/w}
	For $t\geq1$, $k,l,\eta,\xi$ such that $\max(2\sqrt{|\xi|},\sqrt{|\eta|})<t<2\min (|\xi|,|\eta|)$ then 
	\begin{equation}
	\label{dtw/wbuono}
	\frac{\dt w_k(t,\eta)}{w_k(t,\eta)}\frac{w_l(t,\xi)}{\dt w_l(t,\xi)}\lesssim \langle \eta- \xi \rangle.
	\end{equation}
	For all $t\geq 1$, all $k,l$ and $\eta, \xi$ such that $|\eta|\approx |\xi|$, it holds 
	\begin{equation}
	\label{dtw/wgenerale}
	\sqrt{\frac{\dt w_l(t,\xi)}{w_l(t,\xi)}}\lesssim \bigg[\sqrt{\frac{\dt w_k(t,\eta)}{w_k(t,\eta)}}+\frac{|\eta|^{s/2}}{\langle t \rangle^s}\bigg]\langle \eta- \xi \rangle
	\end{equation}
\end{lemma}
\begin{proof}
	As pointed out in remark \ref{remarkdtw}, it is natural to split in two cases since $\dt w$ is close to zero for times near $\sqrt{\eta}$. The proof of this lemma essentially follows by \eqref{dtwR/wR}, let us give some idea.\\ If $t$ not in resonant intervals for both, i.e. $t\not\in I_{k,\eta}\cap I_{l,\xi}$, then \eqref{dtw/wbuono} follows directly by \eqref{dtwR/wR}. If $t$ in resonant intervals for both, use \eqref{dtwR/wR} whit the trichotomy lemma \ref{trichotomy}.\\
	The general case \eqref{dtw/wgenerale} essentially uses \eqref{dtw/wbuono} when possible, and if $t\leq 2\sqrt{\eta}$ then $\frac{|\eta|}{t^2}\gtrsim 1$, hence, since $|\dt w/w\lesssim 1|$, \eqref{dtw/wgenerale} follows. \\
	If $t\geq 2\eta$ then $\dt w_k(t,\eta)=0$. So one consider $|t-|\xi||\leq \frac{1}{K}|\xi|$ and $|t-|\xi||> \frac{1}{K}|\xi|$, and after some computation proves \eqref{dtw/wgenerale}. 
\end{proof}
\begin{lemma}
	\label{lemmawNR/wNR}
	For all $t,\eta, \xi$ we have 
	\begin{equation}
	\label{wNR/wNR}
	\frac{w_{NR}(t,\xi)}{w_{NR}(t,\eta)}\lesssim e^{\mu|\eta-\xi|^{1/2}}
	\end{equation}
\end{lemma}
\begin{proof}
	Proving \eqref{wNR/wNR} is equivalent to 
	\begin{equation*}
		e^{-\mu|\eta-\xi|^{1/2}}\lesssim\frac{w_{NR}(t,\xi)}{w_{NR}(t,\eta)}\lesssim e^{\mu|\eta-\xi|^{1/2}}.
	\end{equation*}
	Essentially it follows from our definition of $w_{NR}$ given in \eqref{defwNR}.
\end{proof}
Finally we recall the most important lemma, which tells us how to exchange frequencies for the multiplier $J$, which appear in \eqref{def:A}. Let us recall its definition 
\begin{equation*}
J_k(t,\eta)=\frac{e^{\mu|\eta|^{1/2}}}{w_k(t,\eta)}+e^{\mu|k|^{1/2}}
\end{equation*} 
In the following we omit the dependence on $t$ for $J$.
\begin{lemma}
	\label{lemmaJ}
	In general it holds 
	\begin{equation}
	\label{J/Jgeneral}
	\frac{J_k(\eta)}{J_l(\xi)}\lesssim \frac{|\eta|}{k^2(1+|t-\frac{\eta}{k}|)} e^{9\mu|k-l,\eta-\xi|^{1/2}}.
	\end{equation}
	If any one of the following holds: ($t\not \in I_{k,\eta}$) or ($k=l$) or ($t\in I_{k,\eta}$ and $\eta\approx \xi$) then we have the improved estimate 
	\begin{equation}
	\label{J/Jimproved}
	\frac{J_{k}(\eta)}{J_l(\xi)}\lesssim e^{10\mu|k-l,\eta-\xi|^{1/2}}.
	\end{equation} 
	Finally, if $t\in I_{l,\xi}$ but $t\not\in I_{k,\eta}$ and $\eta \approx \xi$, then 
	\begin{equation}
	\label{J/Jlxi}
	\frac{J_k(\eta)}{J_l(\xi)}\lesssim \frac{l^2(1+|t-\frac{\xi}{l}|)}{\xi}e^{11\mu|k-l,\eta-\xi|^{1/2}}
	\end{equation}
\end{lemma}
\begin{remark}
 The inequality \eqref{J/Jgeneral} will be used only when $t\in I_{k,\eta}\cap I_{k,\xi}$, but here, thanks to lemma \ref{dtwR/wR} we have 
 \begin{equation}
 \label{J/Jcap}
 \begin{split}
 \frac{J_k(\eta)}{J_l(\xi)}&\lesssim\frac{|\eta|}{k^2}\frac{\dt w_k(t,\eta)}{w_k(t,\eta)}e^{9\mu|k-l,\eta-\xi|^{1/2}}\\
 &\lesssim \frac{|\eta|}{k^2}\sqrt{\frac{\dt w_k(t,\eta)}{w_k(t,\eta)}}\sqrt{\frac{\dt w_l(t,\xi)}{w_l(t,\xi)}}e^{11\mu|k-l,\eta-\xi|^{1/2}},
 \end{split}
 \end{equation}
 where the last one follows by lemma \ref{lemmadtw/w} and absorbing japanese brackets into the exponential thanks to \eqref{appsobexp}.
\end{remark}
\begin{proof}
	Everything follows by definition, properties of $w$ and by the following basic inequality: for $a,b,c,d>0$ it holds
	\begin{equation*}
	\frac{a+b}{c+d}\leq \frac{a}{b}+\frac{c}{d},
	\end{equation*}
	and so 
	\begin{equation*}
	\frac{J_k(\eta)}{J_l(\xi)}\leq \frac{w_l(t,\xi)}{w_k(t,\eta)}e^{\mu|\eta-\xi|^{1/2}}+e^{\mu|k-l|^{1/2}}.
	\end{equation*}
	Then to prove \eqref{J/Jgeneral}, just recall that $w_R=\frac{k^2}{\eta}(1+|\tau|)w_{NR}$. Then $J$ contains $1/w$, hence the most dangerous term is exactly $w_R$, which gives this factor in front of the exponential. Then the factor $9$ is just to absorb all the remaining exponential terms.\\
	The proof of \eqref{J/Jimproved}, follows since we do not have to deal with $w_{R,k}(t,\eta)$ by assumption on the intervals, hence there is no need of the factor in front of the exponential.\\
	Finally \eqref{J/Jlxi}, follows by the same argument of \eqref{J/Jgeneral} since now we are in the resonant intervals for $l,\xi$.
\end{proof}
Finally we state a lemma that is helpful to gain half derivatives in some case. 
\begin{lemma}
	\label{lemmahalfd}
	Let $t\leq \frac{1}{2}\min(\sqrt{|\eta|},\sqrt{|\xi|})$. Then
	\begin{equation}
	\label{J/J-1}
	\bigg|\frac{J_k(\eta)}{J_l(\xi)}-1\bigg|\lesssim \frac{\langle \eta-\xi,k-l\rangle}{\sqrt{|\xi|+|\eta|+|k|+|l|}}e^{11\mu|k-l,\eta-\xi|^{1/2}}
	\end{equation}
\end{lemma}
\begin{proof}
	The proof comes from $|e^x-1|\leq xe^x$, for all the details we refer to \cite[Lemma 3.7]{BM}, The idea is that the term in the l.h.s of \eqref{J/J-1} is something like 
	\begin{align*}
	|e^{\mu(|k|^{1/2}-|l|^{1/2}+|\eta|^{1/2}-|\xi|^{1/2})}-1|&\leq \mu \big(\big||k|^{1/2}-|l|^{1/2}\big|+\big||\eta|^{1/2}-|\xi|^{1/2}\big|\big)e^{\mu|k-l,\eta-\xi|^{1/2}}\\
	&\leq \frac{\langle k-l,\eta-\xi\rangle}{|k|^{1/2}+|l|^{1/2}+|\eta|^{1/2}+|\xi|^{1/2}}e^{\mu|k-l,\eta-\xi|^{1/2}},
	\end{align*}
	where in the last one we have used \eqref{concconc}.\\
\end{proof}
      \section{Elliptic Estimate} 
      In the case of zero mean, the elliptic estimate is the same of the linear case, so very easy to treat. In fact thanks to the assumption $k\neq0$, $\Delta_L^{-1}$ is always well defined, in particular we have the following lemma.
      \begin{lemma}
      	\label{ellipticlemma}
      	Let $u$ and $f$ solutions of \eqref{fu}. Assume that the zero mode is zero, namely $u_0=f_0=0$. Then 
      	\begin{equation}
      	\norma{u}{\mathcal{G}^{\lambda,\sigma-3}}\lesssim \frac{1}{\langle t \rangle^2}\norma{f}{\mathcal{G}^{\lambda,\sigma}}.
      	\end{equation}
      \end{lemma}
  \begin{proof}
  	Proceed as in the linear case, by $\eqref{fu}_2$ we have
  	\begin{align*}
  	\norma{u}{\mathcal{G}^{\lambda,\sigma-3}}^2&\leq \sum_{k\neq 0}\int e^{2\lambda |k,\eta|^s}\frac{|k,\eta|^2\langle k,\eta \rangle^{2(\sigma-3)}}{(k^2+(\eta-kt)^2)^2}|\hf|^2d\eta\\
  	&\lesssim \sum_{k\neq 0}\int e^{2\lambda |k,\eta|^s}\frac{\langle k,\eta \rangle^{2\sigma}}{\langle k,\eta \rangle^{4}\langle \eta-kt\rangle^4}|\hf|^2d\eta\\
  	&\lesssim \frac{1}{\langle t \rangle^4}\norma{f}{\mathcal{G}^{\lambda,\sigma}}^2,
  	\end{align*}
  	where the last one follows just by \eqref{japbound}, i.e. $\langle a\rangle \langle a-b\rangle \gtrsim \langle b \rangle$
  \end{proof}
\begin{remark}
	In the general dynamics, this section becomes involved since essentially just to define the $\Delta_t^{-1}$ that appears in \cite{BM}, one has to have a bound at least on two spatial derivatives of the zero mode. To recover integrability in time there are also other problems to take over. 
\end{remark}
      \section{Transport term}
      This section will be quite the same of \cite{BM}. \\
      Let us start with the bound on $T_N$, by Plancharel one has
      \begin{align*}
      T_N&=\scalar{Af}{A(\uNe \cdot \nabla f_N)-\uNe\cdot \nabla(Af)_N}\\
      &=\scalar{A\hf}{A[\huNe * (|\cdot| \hf_N)]-\huNe* (|\cdot|(A\hf)_N)}\\
      &=i\sum_{k\neq l\neq 0}\int_{\eta,\xi}A_k\bar{\hf}_k(\eta)[A_k(\eta)-A_l(\xi)](l,\xi)\cdot \hu_{k-l}(\eta-\xi)_{<N/8}\hf_l(\xi)_Nd\eta d \xi.
      \end{align*}
       Notice that on the support of the integrand we have that
      \begin{equation}
      \label{ketaapprolxi}
      \big||k,\eta|-|l,\xi|\big|\leq |k-l,\eta-\xi|\leq \frac{3}{2}\frac{N}{8}\leq \frac{3}{16}|l,\xi|,
      \end{equation}
      which implies that $|k,\eta|\approx |l,\xi|$.\\ \\
      Then using the notation introduced in \eqref{notationBC}, we rewrite $T_N$ as follows
      \begin{equation}
      T_N=\scalar{A\hf}{[A-1\circledast A]\huNe * |\cdot| \hf_N}.
      \end{equation}
      So, being $A$ composed essentially by the multiplication of three Fourier multipliers, we want to isolate each of them in $A-1\circledast A$. The idea is that since $u$ is cut at low frequencies, we are interested in isolating the multiplier which acts on $f_N$, by paying exponential regularity on $u_{<N/8}$. In particular the symbol $A-1\circledast A$ can be rewritten as 
      \begin{align*}
      A_k(\eta)-A_l(\xi)=&A_l(\xi)\bigg[\frac{e^{\lambda|k,\eta|^s}J_k(\eta)\langle k,\eta \rangle^{\sigma}}{e^{\lambda|l,\xi|^s}J_l(\xi)\langle l,\xi \rangle^{\sigma}}-1\bigg]\\
      =&A_l(\xi)\big[e^{\lambda|k,\eta|^s-\lambda|l,\xi|^s}-1\big]\\
      &+A_l(\xi)e^{\lambda|k,\eta|^s-\lambda|l,\xi|^s}\bigg[\frac{J_k(\eta)}{J_l(\xi)}-1\bigg]\frac{\langle k,\eta \rangle^{\sigma}}{\langle l,\xi \rangle^{\sigma}}\\
      &+A_l(\xi)e^{\lambda|k,\eta|^s-\lambda|l,\xi|^s}\bigg[\frac{\langle k,\eta \rangle^{\sigma}}{\langle l,\xi \rangle^{\sigma}}-1\bigg]\\
      :=&M_1+M_2+M_3,
      \end{align*}
      where we have just added and subtracted some term. The $M_i$ are of the form $M_i=M_i^1(1\circledast AM_i^2)$. \\
      With the previous splitting, we define 
      \begin{equation*}
      T_N=\sum_{i=1}^{3}\scalar{A\hf}{M_i\big(\huNe * |\cdot| \hf_N\big)}:=T_{N,1}+T_{N,2}+T_{N,3}.
      \end{equation*}
      \subsection{Bound on $T_{N,1}$}
     Recall that on the support of the integrand $|k,\eta|\approx |l,\xi|$.\\
      In $T_{N,1}$ there is only the exponential regularity to control, so by $|e^x-1|<xe^x$, we get 
      \begin{equation}
      \label{inq:TN1part}
      \begin{split}
      |e^{\lambda|k,\eta|^s-\lambda|l,\xi|^s}-1|&\leq\lambda (|k,\eta|^s-|l,\xi|^s)e^{\lambda|k,\eta|^s-\lambda|l,\xi|^s}\\
      &\lesssim \lambda \frac{\big||k,\eta|-|l,\xi|\big|}{|k,\eta|^{1-s}+|l,\xi|^{1-s}}e^{\lambda|k,\eta|^s-\lambda|l,\xi|^s}\\
      &\lesssim \lambda \frac{|k-l,\eta-\xi|}{|k,\eta|^{1-s}+|l,\xi|^{1-s}}e^{c'\lambda|k-l,\eta-\xi|^s}
      \end{split}
      \end{equation}
      for some $c'\in (0,1)$, where we have used \eqref{concconc} and the improved triangular inequality in the exponential, see \eqref{imprevtri}. \\
      Then using the fact that $|k,\eta|\approx |l,\xi|$, and absorbing the numerator of $\eqref{inq:TN1part}_3$ in the exponential, see \eqref{appsobexp}, we have that 
      \begin{equation}
      \label{inq:boundTN1}
      |M_1|= A_l(\xi)|e^{\lambda|k,\eta|^s-\lambda|l,\xi|^s}-1|\lesssim A_l(\xi)\frac{\lambda}{|l,\xi|^{1-s}}e^{c\lambda|k-l,\eta-\xi|^s}= \lambda e^{c\lambda |\cdot|^s}\circledast A|\cdot|^{s-1}
      \end{equation}
      for some $c'<c\in(0,1)$. So now we can bound $T_{N,1}$ in the following way
      \begin{align*}
       |T_{N,1}|&=\big|\Scalar{A\hat{f}}{M_1\big(\huNe*|\cdot|\hf_N\big)}\big|\\
       &\leq \lambda\Scalar{A|\hat{f}|}{\big(e^{c\lambda|\cdot|^s}\huNe*|\cdot|^sA\hat{f}_N\big)}\\
       &\leq \lambda \Scalar{|\cdot|^{s/2}A|\hf|}{\big(e^{c\lambda|\cdot|^{s}}\huNe*|\cdot|^{s/2}A\hat{f}_N\big)},
      \end{align*}
      where in the last inequality we have used the fact that $|k,\eta|\approx |l,\xi|$ to distribute the factor $|l,\xi|^s$. Now we conclude just by Cauchy-Schwartz and Young's inequality, see \eqref{CS+Young}, as follows 
     \begin{equation}
     \label{inq:TN1}
     \begin{split}
     |T_{N,1}|&\leq \lambda \Norma{|\nabla|^{s/2}Af_{\sim N}}{L^2}\Norma{|\nabla|^{s/2}Af_{N}}{L^2}\norma{u_{N/8}}{\mathcal{G}^{\lambda,\sigma-3}}\\
     &\leq \epsilon\frac{\lambda}{\langle t \rangle^2}\Norma{|\nabla|^{s/2}Af_{\sim N}}{L^2}^2,
     \end{split}
     \end{equation}
     where the last one follows by the bootstrap hypothesis, Lemma \ref{ellipticlemma} and properties of Littlewood-Paley decomposition, see \ref{appelitt}. \\
     Notice that the term in \eqref{inq:TN1} is the one that appears in $CK_\lambda$, see \eqref{CKl}, and determine a first condition on the ODE for $\lambda$, defined in \eqref{deflambda}.     \subsection{Bound on $T_{N,3}$}
     The term $T_{N,3}$ is the easiest one, in fact by the mean value theorem, since $|k,\eta|\approx |l,\xi|$,
     \begin{align*}
     M_3&= A_l(\xi)e^{\lambda|k,\eta|^s-\lambda|l,\xi|^s}\bigg[\frac{\langle k, \eta \rangle^\sigma}{\langle l, \xi \rangle^\sigma}-1\bigg]\\
     &\lesssim A_l(\xi)e^{c\lambda |k-l,\eta-\xi|^s}\frac{|k-l,\eta-\xi|}{|l,\xi|}\lesssim e^{\lambda|\cdot|^s}\circledast A|\cdot|^{-1},
     \end{align*}
     where we have again used the concavity property \eqref{concconc} in the exponential and \eqref{appsobexp} to absorb Sobolev regularity. Hence, similarly to $T_{N,1}$,
     \begin{equation}
     \label{bound:TN3}
     \begin{split}
     |T_{N,3}|&=\big|\Scalar{A\hf}{M_3(\huNe*|\cdot|\hf_N)}\big|\lesssim\Scalar{A|\hf|}{e^{\lambda|\cdot|^s}\huNe*A|\hf_N|}\\
     &\lesssim\frac{\epsilon}{\langle t \rangle^2}\norma{Af_{\sim N}}{L^2}^2\lesssim \frac{\epsilon^3}{\langle t \rangle^2}
     \end{split}
     \end{equation}
     where we have used Lemma \ref{ellipticlemma} and bootstrap hypothesis.
     \subsection{Bound on $T_{N,2}$}
     To treat the term $T_{N,2}$ we have to be careful since the multiplier $J$ assign different regularities at different times. Since the multiplier $w_k(t,\eta)$ is constant for $t<\sqrt{\eta}/2$ define the following cut-off 
     \begin{equation}
     \label{def:cutoffchiS}
     \chi^S=\mathds{1}_{t<1/2\min(\sqrt{\xi},\sqrt{\eta})}, \ \ \ \chi^L=1-\chi^S,
     \end{equation}
     and so split $T_{N,2}$ as 
     \begin{equation*}
     T_{N,2}=\Scalar{A\hf}{(\chi^S+\chi^L)M_2\big(\huNe*|\cdot|\hf_N\big)}:=T_{N,2}^S+T_{N,2}^L.
     \end{equation*}
     We start with $T_{N,2}^S$, since it is the simplest one. In fact by lemma \ref{J/J-1} we gain half derivative. Using also the fact that $|k,\eta|\approx |l,\xi|$, we have 
     \begin{align*}
     \chi^SM_2&= \chi^SA_l(\xi)e^{\lambda|k,\eta|^s-\lambda|l,\xi|^s}\bigg[\frac{J_k(\eta)}{J_l(\xi)}-1\bigg]\frac{\langle k,\eta \rangle^{\sigma}}{\langle l,\xi \rangle^{\sigma}}\\
     &\lesssim \chi^SA_l(\xi)e^{c\lambda|k-l,\eta-\xi|^s}\frac{\langle k-l,\eta-\xi\rangle}{|l,\xi|^{1/2}}e^{11\mu|k-l,\eta-\xi|^{1/2}}\\
     &\lesssim\chi^Se^{\lambda|\cdot|^s}\circledast A|\cdot|^{-1/2},
     \end{align*}
     where in the last inequality we have used lemma \ref{lemmaexp} to absorb the numerator. Here it is useful $s>1/2$ to avoid to change the index of regularity $\lambda$. With the last inequality we infer 
     \begin{align}
     \label{TN2Salm}
     |T_{N,2}^S|=|\Scalar{A\hf}{\chi^SM_2\big(\huNe*|\cdot|\hf_N\big)}|\lesssim \Scalar{A|\hf|}{\chi^S\big(e^{\lambda|\cdot|^s}\huNe*|\cdot|^{1/2}A\hf_N\big)},
     \end{align}
     now for $s\geq1/2$, we have that 
     \begin{equation}
     \label{TN2Spart}
     |l,\xi|^{1/2}\lesssim 1+|l,\xi|^s\lesssim 1+|l,\xi|^{s/2}|k,\eta|^{s/2}.
     \end{equation}
     Combining \eqref{TN2Salm} with \eqref{TN2Spart}, with \eqref{CS+Young}, as done previously, we get 
     \begin{equation}
     \label{boundTN2S}
     \begin{split}
     |T_{N,2}^S|\lesssim& \norma{Af_{\sim N}}{L^2}\norma{Af_N}{L^2}\norma{u}{\mathcal{G}^{\lambda,\sigma-3}}\\
     &+\norma{|\nabla|^{s/2}Af_{\sim N}}{L^2}\norma{|\nabla|^{s/2}Af_N}{L^2}\norma{u}{\mathcal{G}^{\lambda,\sigma-3}}\\
     \lesssim& \frac{\epsilon}{\langle t \rangle^2}\big(\norma{Af_{\sim N}}{L^2}^2+\norma{|\nabla|^{s/2}Af_{\sim N}}{L^2}^2\big)
     \end{split}
     \end{equation}
     where we have used \eqref{CS+Young} and the last follows by Lemma \ref{ellipticlemma} and bootstrap hypothesis. \\ \\ \indent
     Now turn to $T_{N,2}^L$. Here, we need to distinguish another case, namely when $t$ is in the resonant interval for $k,\eta$ \textbf{and} $l,\xi$. So define 
     \begin{equation*}
     \chi^D=\mathds{1}_{t\in \mathbf{I}_{k,\eta}\cap \mathbf{I}_{l,\xi}}, \ \ \ \chi^{E}=1-\chi^D,
     \end{equation*} 
     in contrast to \cite{BM}, we do not need $k\neq l$, since if $k=l$ then $\hu(0,\eta-\xi)=0$ (zero mean). So rewrite $T_{N,2}^L$ as follows
     \begin{equation*}
     T_{N,2}^L=\scalar{A\hf}{\chi^L(\chi^D+\chi^E)M_2\big(\huNe*|\cdot|\hf_N\big)}:=T_{N,2}^D+T_{N,2}^E.
     \end{equation*}
     For convenience, let $M_2^D=\chi^L\chi^DM_2$, and analogously $M_2^E$. \\
     To bound $M_2^D$, we use lemma \ref{lemmaJ}, since we cannot gain much by the factor $-1$. So we have
     \begin{align*}
     |M_2^D|&\lesssim\chi^L\chi^DA_l(\xi)e^{c\lambda|k-l,\eta-\xi|^s}\frac{J_k(\eta)}{J_l(\xi)}\\
     &\lesssim \chi^L\chi^DA_l(\xi)e^{c\lambda|k-l,\eta-\xi|^s} \frac{|\eta|}{k^2}\sqrt{\frac{\dt w_k(t,\eta)}{w_k(t,\eta)}}\sqrt{\frac{\dt w_l(t,\xi)}{w_l(t,\xi)}}e^{20\mu|k-l,\eta-\xi|^{1/2}},
     \end{align*}
     where we have applied \eqref{J/Jcap}.
     Now, on the support of $M_2^D$, $1<t\approx \eta/k$ so 
     \begin{equation*}
     |M_2^D| \lesssim 	\chi^L\chi^D\frac{t^2}{|\eta|} \sqrt{\frac{\dt w}{w}}\bigg[e^{\lambda|\cdot|^s}\circledast \sqrt{\frac{\dt w}{w}}A\bigg]
     \end{equation*}
     where we have used the same properties of previous cases for the exponential. In the support of of $M_2^D$, namely $k<\sqrt{\eta}$. Hence $|l,\xi|\approx|k,\eta|\lesssim |\eta|$, with this observation we get 
     \begin{equation*}
\begin{split}
     |T_{N,2}^D|&=|\scalar{A\hf}{M_2^D\big(\huNe*|\cdot|\hf_N\big)}|\\
     &\lesssim t^2\big \langle\chi^L\chi^D\sqrt{\frac{\dt w}{w}}A|\hf|,\big(e^{\lambda|\cdot|^s}\huNe*\sqrt{\frac{\dt w}{w}}A\hf_N\big)\big \rangle,
\end{split}
     \end{equation*}
     where the observation is used in the second inequality to drop the term $|\cdot|$ in front of $\hf_N$. Observing that in the support of the integrand $|k|\leq |\eta|$, we have also $A\lesssim \tilde{A}$. Hence we conclude the estimates by using \eqref{CS+Young} and bootstrap to get 
     \begin{equation}
     \label{boundTN2D}\begin{split}
     |T_{N,2}^D|&\lesssim t^2 \frac{\epsilon}{\langle t \rangle^2}\bigg \lVert\sqrt{\frac{\dt w}{w}}\tilde{A}f_{\sim N}\bigg \rVert_{L^2}\bigg \lVert\sqrt{\frac{\dt w}{w}}\tilde{A}f_{N}\bigg \rVert_{L^2}\\
     &\lesssim \epsilon\bigg \lVert\sqrt{\frac{\dt w}{w}}\tilde{A}f_{\sim N}\bigg \rVert_{L^2}^2,  
     \end{split}       
     \end{equation}
     which is a term like the one of $CK_w$, see \eqref{CKw}.\\ \indent
     It remains to treat $T_{N,2}^{E}$. Here, we are not in resonant interval for $k,\eta$ and $l,\xi$, but still we cannot use \eqref{J/J-1} to gain half derivative. Hence we have to split w.r.t. the relative size of $l,\xi$. Namely
     \begin{equation*}
     T_{N,2}^E=\scalar{A\hf}{\chi^L\chi^E(\mathds{1}_{|l|>100|\xi|}+\mathds{1}_{|l|\leq 100|\xi|})M_2\big(\huNe *|\cdot|\hf_N\big)}:=T_{N,2}^{E,z}+T_{N,2}^{E,y}.
     \end{equation*}
     On the support of $T_{N,2}^{E,z}$, we have that $|\eta|<\frac{1}{3}|l|$, hence by definition of $J$ and $w$, we have 
     \begin{align*}
     \bigg|\frac{J_k(\eta)}{J_l(\xi)}-1\bigg|&=\bigg|\frac{w_k(t,\eta)^{-1}e^{\mu|\eta|^{1/2}}+e^{\mu|k|^{1/2}}}{w_l(t,\xi)^{-1}e^{\mu|\xi|^{1/2}}+e^{\mu|l|^{1/2}}}-1\bigg|\\
     &\lesssim e^{2\mu |\eta|^{1/2}-\mu|l|^{1/2}}+|e^{\mu|k|^{1/2}-\mu|l|^{1/2}}-1|\\
     &\lesssim \frac{1}{|l|^{1/2}}+\frac{|k-l|}{|k|^{1/2}+|l|^{1/2}}e^{\mu|k-l|^{1/2}}.
     \end{align*}
     Now, since $|l,\xi|\lesssim |l|$, and absorbing properly exponential terms with \eqref{appexp}, we get that 
     \begin{align*}
     |T_{N,2}^{E,z}|&=|\scalar{A\hf}{\chi^L\chi^E\mathds{1}_{|l|>100|\xi|}M_2\big(\huNe*|\cdot|\hf_N\big)}|\\
     &\lesssim \scalar{A|\hf|}{\chi^L\chi^E\mathds{1}_{|l|>100|\xi|}\big(e^{\lambda|\cdot|^s}\huNe*|\cdot|^{1/2}\hf_N\big)}\\
     &\lesssim \scalar{|\cdot|^{s/2}A|\hf|}{\chi^L\chi^E\mathds{1}_{|l|>100|\xi|}\big(e^{\lambda|\cdot|^s}\huNe*|\cdot|^{s/2}\hf_N\big)},
     \end{align*} 
     where in the last one we use the fact that $|l,\xi|\approx |k,\eta|$ and $s>1/2$. Then, by \eqref{CS+Young} and bootstrap we infer 
     \begin{equation}
     \label{boundTN2Ez}
     |T_{N,2}^{E,z}|\lesssim \frac{\epsilon}{\langle t\rangle^2}\Norma{|\nabla|^{s/2}Af_{\sim N}}{L^2}^2.
     \end{equation}
     Finally, on the support of $T_{N,2}^{E,y}$, we have that $|\eta| \approx |\xi|$ and we can apply \eqref{J/Jimproved}. Hence we have that 
     \begin{align*}
     |T_{N,2}^{E,y}|\lesssim \scalar{A|\hf|}{\chi^L\chi^E\mathds{1}_{|l|<100|\xi|}\big(e^{\lambda|\cdot|^s}\huNe*|\cdot|\hf_N\big)},
     \end{align*}
     then just observe that $|l,\xi|\lesssim |\xi|<t^2$. Since $|k,\eta|\approx|l,\xi|$, one has 
     \begin{equation*}
     |l,\xi|\lesssim |l,\xi|^{s}t^{2-2s}\approx|k,\eta|^{s/2}|l,\xi|^{s/2}t^{2-2s},
     \end{equation*}
     hence by previous arguments and bootstrap
     \begin{equation}
     \label{boundTN2Ey}
     \begin{split}
     |T_{N,2}^{E,y}|&\lesssim t^{2-2s} \scalar{|\cdot|^{s/2}A|\hf|}{\chi^L\chi^E\mathds{1}_{|l|<100|\xi|}\big(e^{\lambda|\cdot|^s}\huNe*|\cdot|^{s/2}\hf_N\big)}\\
     &\lesssim \frac{\epsilon}{\langle t \rangle^{2s}}\Norma{|\nabla|^{s/2}Af_{\sim N}}{L^2}^2.
     \end{split}
     \end{equation}
     Here it is important that $s>1/2$ to have integrability in time.\\ \indent
     Combining \eqref{inq:TN1}, \eqref{bound:TN3}, \eqref{boundTN2S}, \eqref{boundTN2D}, \eqref{boundTN2Ez} and \eqref{boundTN2Ey}, we get 
     \begin{equation}
     |T_N|\lesssim \epsilon\bigg(\frac{\lambda+1}{\langle t \rangle^2}+\frac{1}{\langle t \rangle ^{2s}}\bigg)\Norma{|\nabla|^{s/2}Af_{\sim N}}{L^2}^2+\epsilon\bigg \lVert\sqrt{\frac{\dt w}{w}}\tilde{A}f_{\sim N}\bigg \rVert_{L^2}^2 + \frac{\epsilon}{\langle t \rangle^2}\Norma{Af_{\sim N}}{L^2}^2.
     \end{equation}
     Hence by summing up in $N$, thanks to basic properties of Littlewood-Paley decomposition, see \ref{appelitt}, and bootstrap hypothesis, we have that 
     \begin{equation}
     |T|\lesssim \epsilon \bigg(\frac{\lambda+1}{\langle t \rangle^2}+\frac{1}{\langle t \rangle ^{2s}}\bigg)\Norma{|\nabla|^{s/2}Af}{L^2}^2+\epsilon\bigg \lVert\sqrt{\frac{\dt w}{w}}\tilde{A}f\bigg \rVert_{L^2}^2 + \frac{\epsilon^3}{\langle t \rangle^2},
     \end{equation} 
     hence $\lambda$ as in \eqref{deflambda}, namely
     \begin{equation*}
     \dot{\lambda}(t)=-\frac{\delta_\lambda}{\langle t \rangle^{2s}}(\lambda(t)+1)
     \end{equation*}
     we prove proposition \ref{proptransport}. \qed
     
     \section{Reaction term}
     The reaction term is the most challenging one. In fact the weight was built to predict the worst possible case of it. \\
     In the treatment of this term, we have major simplifications due to the assumption of zero mean. In fact $\Delta_L^{-1}$ is always well defined, and we do not have remainders terms created by the change of variables. \\ \indent
     So, recall that we keep the variables $(l,\xi)$ for high frequencies, that now are on $u$ instead of $f$. This is just a change of variable in the convolution on the Fourier side. Writing explicitly $R_N$, we have that 
     \begin{align*}
     R_N=&i\sum_{k\neq l \neq 0} \int_{\eta,\xi} A_k(\eta)\bar{\hf}_k(\eta)A_k(\eta) \hu_l(\xi)_N\cdot (k-l,\eta-\xi)\hf_{k-l}(\eta-\xi)_{<N/8}d\eta d\xi \\
     &-\scalar{A\hf}{\big(\hu_N* |\cdot| (A\hf)_{<N/8}\big)},
     \end{align*}
     Now, recall $\eqref{fu}_2$, i.e. $u=\nabla_{z,y}^\perp \phi$. Observe that $(-\xi,l)\cdot(k-l,\eta-\xi)=(k,\eta)\cdot(-\xi,l)$, which with our notation reads as $|\cdot|^\perp \circledast |\cdot|=|\cdot|\cdot (|\cdot|^\perp \circledast 1)$. The previous equality essentially is Leibniz rule, since $\div \nabla^\perp=0$.\\
     Hence, $R_N$ can be written as 
     \begin{align*}
     R_N=&\scalar{A\hf}{|\cdot|\cdot A\big(|\cdot|^\perp  \hphi_N* \hfNe\big)}\\
     &-\scalar{A\hf}{\big(\hu_N* |\cdot| (A\hf)_{<N/8}\big)}:=R_{N}^1+R_N^3,
     \end{align*}
     where we denote as $R_N^3$ the second term to keep the notation of \cite[Section 6]{BM}. The main contribution is given when derivatives hits the velocity, hence $R_N^1$. Instead $R_N^3$ is a commutator term which is easy to treat, so let us start with that.
     \subsection{Bound on $R_N^3$} 
     For $R_N^3$ we directly apply \eqref{CS+Young}. Also we use that on the support of the integrand $|k-l,\eta-\xi|\lesssim |l,\xi|$ (recall again that now $|l,\xi|$ are the frequencies for $u$), to conclude that 
     \begin{align*}
     |R_N^3|=& |\scalar{A\hf}{\big(\hu_N* |\cdot| (A\hf)_{<N/8}\big)}|\\
     \lesssim& \scalar{A|\hf|}{\big||\cdot|\hu_N*(A\hf_{<N/8})\big|}\\
     \lesssim & \norma{Af_{\sim N}}{L^2}\norma{u_N}{H^{\sigma-4}}\norma{Af_{<N/8}}{L^2}\\
     \lesssim& \frac{\epsilon}{\langle t\rangle^2}\norma{Af_{\sim N}}{L^2}^2,
     \end{align*}
     where we have used also lemma \ref{ellipticlemma} and bootstrap. Then using Littlewood-Paley decomposition properties and bootstrap again, we infer 
     \begin{equation}
     \label{boundRN3}
     \sum_{N\geq 8}|R_N^3|\lesssim  \frac{\epsilon^3}{\langle t \rangle^2}
     \end{equation}
     which is a term that appears in proposition \ref{propreaction}. \qed
     \subsection{Bound on $R_N^1$}
     Here the multiplier $A$ hits the velocity $u$, or equivalently $\phi$. Since $A$ assign different regularity at different times, due to the presence of the weight $w$, it is natural to split $R_N^1$ to isolate time intervals with different regularities. \\
     In particular, since on the support of the integrand $|k,\eta|\approx|l,\xi|$, it is relevant to see when they are resonant or not, because in resonant intervals we cannot recover integrability in time. So define 
     \begin{align*}
     1&=\mathds{1}_{t\not \in I_{k,\eta}, t \not \in I_{l,\xi}}+\mathds{1}_{t\not \in I_{k,\eta}, t  \in I_{l,\xi}}+\mathds{1}_{t \in I_{k,\eta}, t \not \in I_{l,\xi}}+\mathds{1}_{t\in I_{k,\eta}, t \in I_{l,\xi}}\\
     &:=\chi^{NR,NR}+\chi^{NR,R}+\chi^{R,NR}+\chi^{R,R},
     \end{align*}
     where 'NR' and 'R' stands for \textit{Non-Resonant} and \textit{Resonant} respectively. The first apex refer to frequencies $(k,\eta)$, the second one to $(l,\xi)$. \\
     So rewrite $R_N^1$ as 
     \begin{align*}
     R_N^1&=-\scalar{ \big(\chi^{NR,NR}+\chi^{NR,R}+\chi^{R,NR}+\chi^{R,R}\big)A\hf}{|\cdot|\cdot A\big(|\cdot|^\perp  \hphi_N* \hfNe\big)}\\
     &:=R_{N}^{NR,NR}+R_N^{NR,R}+R_N^{R,NR}+R_N^{R,R}.     
     \end{align*}
     \subsubsection{Bound on $R_N^{NR,NR}$}
     Recall that it always holds $|k,\eta|\approx |l,\xi|$. Then observe that since we are in non resonant intervals, we can use \eqref{J/Jimproved}, hence
     \begin{align*}
     \chi^{NR,NR}A_k(\eta)=& \chi^{NR,NR} A_l(\xi) \frac{J_k(\eta)}{J_l(\xi)}e^{\lambda|k,\eta|^s-|l,\xi|^s}\frac{\langle k,\eta \rangle^\sigma}{\langle l,\xi \rangle^\sigma}\\
     &\lesssim \chi^{NR,NR}A_l(\xi)e^{10\mu|k-l,\eta-\xi|^{1/2}}e^{c\lambda|k-l,\eta-\xi|^s}\\
     & \lesssim \chi^{NR,NR}A\circledast e^{\lambda|\cdot|}
     \end{align*}
     where we have used the usual properties of the exponential \eqref{appexp}.\\
     In the terms of $R_N^1$, it appears also the term $(k,\eta)\cdot(-\xi,l)=\eta l-k\xi$, let us give a bound on that 
     \begin{equation}
     \label{eqcdot}
\begin{split}
     \big||\cdot|\cdot(|\cdot|^\perp \circledast 1)\big|&=|\eta l-k\xi|\leq |\eta(l-k)|+|k(\eta-\xi)|\\
     &\lesssim |k,\eta||k-l,\eta-\xi|\\
     &=\big||\cdot|\cdot(1\circledast |\cdot|)\big|
\end{split}
     \end{equation}
     Then, using previous inequalities, we have that 
     \begin{align*}
     |R_N^{NR,NR}|&=|\scalar{\chi^{NR,NR}A\hf}{A|\cdot|\cdot \big(|\cdot|^{\perp}\hphi_N*\hf_{<N/8}\big)}\\
     &\lesssim\scalar{\chi^{NR,NR}|\cdot|A|\hf|}{\big(A\hphi_N*|\cdot|e^{\lambda|\cdot|}\hf_{<N/8}\big)}\\
     &\approx \scalar{\chi^{NR,NR}|\cdot|^{s/2}A|\hf|}{\big(|\cdot|^{1-s/2}A\hphi_N*|\cdot|e^{\lambda|\cdot|}\hf_{<N/8}\big)}
     \end{align*}
      where we use the fact that $|k,\eta|\approx |l,\xi|^{1-s/2}|k,\eta|^{s/2}$.\\
      So thanks to \eqref{CS+Young}, we conclude that 
      \begin{align*}
      |R_N^{NR,NR}|&\lesssim \Norma{|\nabla|^{s/2}Af_{\sim N}}{L^2}\Norma{\chi^{NR}|\nabla|^{1-s/2}A\phi_N}{L^2}\norma{f}{\mathcal{G^{\lambda,\sigma}}}\\
      &\lesssim \epsilon\Norma{|\nabla|^{s/2}Af_{\sim N}}{L^2}\Norma{\chi^{NR}|\nabla|^{1-s/2}A\phi_N}{L^2},
      \end{align*}
      where in the last one we have used also bootstrap hypothesis. The multiplier $\chi^{NR}$ means that $t$ is not in resonant intervals.\\
      Now we exploit the fact that $\phi=\Delta_L^{-1}f$ to recover some integrability in time. Here it is crucial that we are not in resonant intervals, otherwise we cannot recover any integrability. In particular, we claim that 
      \begin{equation}
      \label{nabla/deltaL}
      \mathds{1}_{t\not \in I_{l,\xi}} \frac{|l,\xi|^{1-s/2}}{l^2+|\xi-lt|^2}\lesssim \frac{|l,\xi|^{s/2}}{\langle t \rangle^{2s}},
      \end{equation}
      if we are able to prove that, then we conclude the bound on $R_N^1$, in fact 
      \begin{equation*}
      \norma{|\nabla|^{1-s/2}A\phi_N}{L^2}=\norma{|\nabla|^{1-s/2}|\Delta_L|^{-1}Af_N}{L^2}\lesssim \frac{1}{\langle t \rangle^{2s}}\norma{|\nabla|^{s/2}Af}{L^2}.
      \end{equation*}
      So the bound on $R_N^1$ becomes 
      \begin{equation}
      \label{boundRNNRNR}
      |R_N^{NR,NR}|\lesssim \frac{\epsilon}{\langle t \rangle^{2s}}\norma{|\nabla|^{s/2}Af_{\sim N}}{L_2}^2.
      \end{equation} 
      Then let us prove \eqref{nabla/deltaL}.
     \begin{proof}[Proof of \eqref{nabla/deltaL}]
     	 In general, outside resonant interval, i.e $t\not \in I_{l,\xi}$, one has $|\xi/l-t|\gtrsim \xi/l^2$, or equivalently $|\xi-lt|\gtrsim \xi/l$ (see definition \eqref{defIketa}). \\
      Now let $\frac{1}{2}|lt|\leq |\xi|\leq 2|lt|$. Here, $|\xi/l-t|\gtrsim t$, hence
      \begin{equation*}
      \frac{|l,\xi|^{1-s}}{l^2+|\xi-lt|^2}\lesssim \frac{|l,\xi|^{1-s}}{l^2+t^2}\lesssim \frac{|l|^{1-s}\langle t\rangle^{1-s}}{l^{1-s}|t|^{1+s}}\lesssim \frac{1}{\langle t \rangle^{2s}},
      \end{equation*}
      where we use the fact that $a^2+b^2\gtrsim a^{1-s}b^{1+s}$, for $s\in(0,1)$ and $a,b$ positive.\\
      If $|\xi|\leq |lt|/2$, one has that $|\xi-lt|\gtrsim |lt|$, hence
      \begin{equation*}
      \frac{|l,\xi|^{1-s}}{l^2+|\xi-lt|^2} \lesssim \frac{|l,\xi|^{1-s}}{|lt|^2}\lesssim \frac{|lt|^{1-s}}{\langle lt\rangle^{1+s}}\lesssim \frac{1}{\langle t \rangle^{2s}}.
      \end{equation*}
      Finally, if $|\xi|\geq 2|lt|$, one has that 
      \begin{equation*}
      \frac{|l,\xi|^{1-s}}{l^2+|\xi-lt|^2} \lesssim\frac{|\xi|^{1-s}}{\langle \xi \rangle^2}\lesssim \frac{1}{\langle t \rangle^{1+s}}\lesssim \frac{1}{\langle t \rangle^{2s}},
      \end{equation*}
      where the last follows since $s<1$. This finishes the proof of \eqref{nabla/deltaL}.
     \end{proof}
 \subsubsection{Bound on $R_N^{NR,R}$}
 In this case we have that $(k,\eta)$ is non resonant and $(l,\xi)$ is resonant. By definition of $I_{l,\xi}$, implies that $|l|^2\lesssim |\xi|$, and since $|l,\xi|\approx N$, this means that $|\xi|\approx N$ and consequently $|\eta|\approx|\xi|$. So, since $(l,\xi)$ are resonant, we are sure that $\dt w_l(t,\xi)\neq  0$. We can rewrite \eqref{dtw/wgenerale} as follows
 \begin{equation}
 \label{pass1RN1NRR}
 1\lesssim\sqrt{\frac{ w_l(t,\xi)}{\dt w_l(t,\xi)}}  \bigg[\sqrt{\frac{\dt w_k(t,\eta)}{w_k(t,\eta)}}+\frac{|k,\eta|^{s/2}}{\langle t \rangle^s}\bigg]\langle \eta- \xi \rangle,
 \end{equation} 
 As done previously we want to exchange $A_k(\eta)$ with $A_l(\xi)$. On the support of the integral we can apply \eqref{J/Jlxi}, hence 
 \begin{align*}
 \chi^{NR,R}A_k(\eta)=&\chi^{NR,R} A_l(\xi) \frac{J_k(\eta)}{J_l(\xi)}e^{\lambda|k,\eta|^s-|l,\xi|^s}\frac{\langle k,\eta \rangle^\sigma}{\langle l,\xi \rangle^\sigma}\\
 \lesssim&\chi^{NR,R}A_l(\xi)e^{11\mu|k-l,\eta-\xi|^{1/2}}e^{c\lambda|k-l,\eta-\xi|^{s}}\frac{l^2(1+|t-\xi/l|)}{|\xi|}\\
 \lesssim&\chi^{NR,R}A_l(\xi)e^{\lambda|k-l,\eta-\xi|^{s}}\frac{w_R(\xi)}{w_{NR}(\xi)}\\
 =&\chi^{NR,R}\big(A\frac{w_R}{w_{NR}}\circledast e^{c'\lambda|\cdot|}\big),
 \end{align*}
 where we have also used the fact that $w_R(\xi)\approx \displaystyle \frac{l^2}{\xi}\bigg(1+|t-\frac{\xi}{l}|\bigg)w_{NR}(\xi)$, and usual properties of the exponential. \\
 Since we want to use \eqref{pass1RN1NRR}, we absorb the coefficient $\langle \eta-\xi\rangle$ in the exponential applied to $\hfNe$. Also, remember that $|k,\eta|\approx|l,\xi|$, so \eqref{eqcdot} is equivalent to $\big||\cdot|^\perp\circledast|\cdot|\big|$ (it is just the original form of $R_N$). So we have that
 \begin{align*}
 |R_N^{NR,R}|&=|\scalar{\chi^{NR,R}A\hf}{A\big(|\cdot|^\perp\hphi_N*|\cdot|\hf_{<N/8}}|\\
 &\lesssim\bigg\langle \chi^{NR,R}\bigg[\sqrt{\frac{\dt w}{w}}+\frac{|\cdot|^{s/2}}{\langle t \rangle^s}\bigg]A\hf,\bigg[\sqrt{\frac{w}{\dt w}}|\cdot|A\frac{w_R}{w_{NR}}\hphi_N\bigg]*e^{\lambda|\cdot|}\hf_{<N/8}\bigg \rangle \\
 &\lesssim \epsilon\bigg(\bigg \lVert \sqrt{\frac{\dt w}{w}} \tilde{A}f_{\sim N}\bigg \rVert_{L^2} +\frac{1}{\langle t \rangle^s}\Norma{|\nabla|^{s/2}Af_{\sim N}}{L_2}\bigg)\bigg \lVert \mathds{1}_{t\in I_{l,\xi}}\sqrt{\frac{ w}{\dt w}}|\nabla|\frac{w_R}{w_{NR}}\tilde{A}\phi_N \bigg \rVert,
 \end{align*}
 where in the last one we have the fact that $|k|<|\eta|$ to exchange $A$ with $\tilde{A}$, inequality \eqref{CS+Young} and bootstrap hypothesis to replace $\norma{f}{\mathcal{G^{\lambda,\sigma}}}$ with $\epsilon$.\\
 Then applying $ab/2\leq a^2+b^2$, we have that 
 \begin{align*}
 |R_N^{NR,R}|\lesssim& \epsilon\bigg \lVert \sqrt{\frac{\dt w}{w}} \tilde{A}f_{\sim N}\bigg \rVert_{L^2}^2 +\frac{\epsilon}{\langle t \rangle^{2s}}\Norma{|\nabla|^{s/2}Af_{\sim N}}{L_2}^2\\
 &+ \epsilon\bigg \lVert \mathds{1}_{t\in I_{l,\xi}}\sqrt{\frac{ w}{\dt w}}|\nabla|\frac{w_R}{w_{NR}}\Delta_L^{-1}\tilde{A}f_N \bigg \rVert_{L^2}^2,
 \end{align*}
 where again we have used zero mean condition, $l\neq0$, to write $\Delta_L^{-1}$.\\
 Now we claim that 
 \begin{equation}
 \label{replacedtww}
 M^{R}:=\mathds{1}_{t\in I_{l,\xi}}\sqrt{\frac{w_l(t,\xi)}{\dt w_l(t,\xi)}}|l,\xi|\frac{w_R(t,\xi)}{w_{NR}(t,\xi)}\frac{1}{l^2+|\xi-tl|^2}\lesssim \sqrt{\frac{\dt w_l(t,\xi)}{w_l(t,\xi)}}.
 \end{equation}
 If it is true, we conclude the estimate for $R_{N}^{NR,R}$ as follows 
 \begin{equation}
 \label{boundRNNRR}
|R_N^{NR,R}|\lesssim \epsilon\bigg \lVert \sqrt{\frac{\dt w}{w}} \tilde{A}f_{\sim N}\bigg \rVert_{L^2}^2 +\frac{\epsilon}{\langle t \rangle^{2s}}\Norma{|\nabla|^{s/2}Af_{\sim N}}{L_2}^2.
 \end{equation}
 Let us prove \eqref{replacedtww}. By definition of $w_R, w_{NR}$ and lemma \ref{lemma:dtw/wapprox}, we have that
 \begin{align*}
 M^R&\lesssim |\xi|\bigg(\sqrt{1+\bigg|t-\frac{\xi}{l}\bigg|}\bigg)\frac{l^2}{|\xi|}\bigg(1+\bigg|t-\frac{\xi}{l}\bigg|\bigg)\frac{1}{l^2(1+|t-\xi/l|^2)}\\
 &\lesssim \frac{1}{(1+|t-\xi/l|)}\approx\sqrt{\frac{\dt w_l(t,\xi)}{ w_l(t,\xi)}}.
 \end{align*}
 So we have proved the bound on $R_N^{NR,R}$. \qed
 \subsubsection{Bound on $R_N^{R,NR}$}
 This term is the most dangerous one, in particular the weight $w$ was built to mimic its growth. First of all, $k,\eta$ are resonant and $l,\xi$ non resonant, but can happen that $k,\xi$ is also resonant ($k\neq l$). So it is natural to split again $R_N^{R,NR}$ in the following way 
 \begin{align*}
 R_N^{R,NR}=&\scalar{\chi^{R,NR}\big(\mathds{1}_{t\in I_{k,\xi}}+\mathds{1}_{t\not \in I_{k,\xi}}\big)A\hf}{A\big(|\cdot|^\perp \hphi_N*|\cdot|\hf_{<N/8}}\\
 :=&R_N^{R,NR:D}+R_N^{R,NR;E},
 \end{align*}
 where D stands for 'difficult' and E for 'easy'.\\
 As always, the main point is exchanging $J_k(\eta)$ with $J_l(\xi)$, where one pays different prices depending on which situation you are. \\ \indent
 Let us start with $|R_N^{R,NR;D}|$. First of all, on the support of the integrand $|\eta|\approx|\xi|$. We have to use \eqref{J/Jcap}, so, omitting time cut off functions, we have that 
 \begin{align*}
  A_k(\eta)&=A_l(\xi) \frac{J_k(\eta)}{J_l(\xi)}e^{\lambda|k,\eta|^s-|l,\xi|^s}\frac{\langle k,\eta \rangle^\sigma}{\langle l,\xi \rangle^\sigma}\\
  &\lesssim A_l(\xi) \frac{|\eta|}{k^2}\sqrt{\frac{\dt w_k(t,\eta)}{w_k(t,\eta)}}\sqrt{\frac{\dt w_l(t,\xi)}{w_l(t,\xi)}} e^{\lambda|k-l,\eta-\xi|^s}
 \end{align*}
 where as always we have absorbed all the exponential terms in the one that remain. Now, since $|\eta| \approx |\xi|$, and in general holds that $\displaystyle \frac{l^2}{k^2}\leq \langle k-l\rangle^2$, we have that
 \begin{align*}
 \chi^{R,NR}\mathds{1}_{t\in I_{k,\xi}}A_k(\eta)\lesssim \chi^{R,NR}\mathds{1}_{t\in I_{k,\xi}} \sqrt{\frac{\dt w}{w}}\bigg(\frac{|\dy|}{|\dz|^2}\sqrt{\frac{\dt w}{w}}A\circledast\langle \cdot \rangle^2 e^{\lambda|\cdot|}\bigg).
 \end{align*} 
 So we proceed as follows
 \begin{align*}
 |R_N^{R,NR;D}|&=\big|\scalar{\mathds{1}_{t\in I_{k,\xi}}\chi^{R,{NR}}A\hf}{A\big(|\cdot|^\perp \hphi_N*|\cdot|\hfNe\big)}\big|\\
 &\lesssim \bigg \langle\mathds{1}_{t\in I_{k,\xi}}\chi^{R,NR}\sqrt{\frac{\dt w}{w}}A|\hf|, \bigg|\frac{|\dy|}{|\dz|^2}\sqrt{\frac{\dt w}{w}}A\hphi_N*\langle\cdot \rangle^\sigma e^{\lambda|\cdot |^s}\hfNe\bigg| \bigg\rangle\\
 &\lesssim \bigg \lVert \sqrt{\frac{\dt w}{w}}\tilde{A}f_{\sim N}\bigg \rVert_{L^2}\bigg \lVert \mathds{1}_{t\not \in I_{l,\xi}}\frac{|\dy|}{|\dz|^2}\Delta_L^{-1}\sqrt{\frac{\dt w}{w}}  \tilde{A}f_N \bigg \rVert_{L^2} \norma{f}{\mathcal{G^{\lambda,\sigma}}},
  \end{align*}
 where we have used \eqref{CS+Young} and the fact that on the support of the integrand $|k|\lesssim |\eta|, |l|\lesssim |\xi|$ to replace $A$ with $\tilde{A}$. Now we claim that 
 \begin{equation}
 \label{claimNR}
 \mathds{1}_{t\not \in I_{l,\xi}}\frac{|\xi|}{|l|^2}\frac{1}{l^2+|\xi-lt|^2}\lesssim 1.
 \end{equation}
 To prove \eqref{claimNR}, just observe that we are in non resonant interval for $(l,\xi)$, then $|\xi-lt|\gtrsim \xi/l$, hence the claim follows.\\
 So we have completed the bound on $R_N^{R,NR;D}$, in fact thanks to bootstrap we conclude that 
 \begin{equation}
 \label{boundRNRNRD}
 |R_N^{R,NR;D}|\lesssim \epsilon\bigg \lVert \sqrt{\frac{\dt w}{w}}\tilde{A}f_{\sim N}\bigg \rVert_{L^2}^2.
 \end{equation}
To bound $R_N^{R,NR;E}$, we can apply directly \eqref{J/Jimproved} to get that 
\begin{equation*}
\chi^{R,NR}\mathds{1}_{t\not \in I_{k,\xi}}A_k(\eta)\lesssim \chi^{R,NR}\mathds{1}_{t\not \in I_{k,\xi}}\big(A\circledast e^{\lambda|\cdot|^s}\big).
\end{equation*}
Using also that $|l,\xi|\approx |k,\eta|$, we get that 
\begin{align*}
|R_N^{R,NR;E}|&=\big|\scalar{\chi^{R,NR}\mathds{1}_{t\not \in I_{k,\xi}}A\hf}{A\big(|\cdot|^\perp \hphi_N*|\cdot|\hf_{<N/8}}\big|\\
&\lesssim \scalar{\chi^{R,NR}\mathds{1}_{t\not \in I_{k,\xi}}A|\cdot|^{s/2}|\hf|}{\big||\cdot|^{1-s/2}A \hphi_N*|\cdot|^\sigma e^{\lambda|\cdot|^s}\hf_{<N/8}\big|}\\
&\lesssim \epsilon \Norma{|\nabla|^{s/2}Af_{\sim N}}{L^2}\Norma{\chi^{NR}|\nabla|^{1-s/2}\Delta_L^{-1}Af_N}{L^2},
\end{align*}
where we have also used bootstrap hypothesis. But it appears exactly the same term as in the $R_N^{NR,NR}$, hence using \eqref{nabla/deltaL}, we infer that 
\begin{equation}
\label{boundRNRNRE}
|R_N^{R,NR;E}|\lesssim \epsilon \Norma{|\nabla|^{s/2}Af_{\sim N}}{L^2}^2.
\end{equation}
Putting together \eqref{boundRNRNRD} and \eqref{boundRNRNRE}, we have that
\begin{equation}
\label{boundRNRNR}
|R_N^{R,NR}|\lesssim \epsilon\bigg \lVert \sqrt{\frac{\dt w}{w}}\tilde{A}f_{\sim N}\bigg \rVert_{L^2}^2+\epsilon \Norma{|\nabla|^{s/2}Af_{\sim N}}{L^2}^2.
\end{equation}
With this bound we have completed the bound for $R_N^{R,NR}$. \qed
\subsubsection{Bound on $R_N^{R,R}$}
Finally it remains to bound the last term of $R_N^1$, namely $R_N^{R,R}$. Here both $(k,\eta)$ and $(l,\xi)$ are resonant. Let us rewrite $R_N^{R,R}$ for clarity 
\begin{equation*}
R_N^{R,R}=\scalar{\chi^{R,R}A\hf}{A\big(|\cdot|^\perp \hphi_N*|\cdot|\hfNe\big)}
\end{equation*}
where we have used the original formulation given in $R_N$ (or Leibniz rule). In this case, ignoring the perp since everything will be with absolute values, we have that 
\begin{align*}
\chi^{R,R}A(|\cdot|\circledast |\cdot|)&=\chi^{R,R}A_k(\eta)(|l,\xi||k-l,\eta-\xi|)\\
&=\chi^{R,R}A_l(\xi)\frac{J_k(\eta)}{J_l(\xi)}e^{\lambda|k,\eta|^s-|l,\xi|^s}\frac{\langle k,\eta \rangle^\sigma}{\langle l,\xi \rangle^\sigma}|l,\xi||k-l,\eta-\xi|\\
&\lesssim \chi^{R,R}A_l(\xi)e^{c\lambda|k-l,\eta-\xi|^s}|l,\xi||k-l,\eta-\xi|\frac{J_k(\eta)}{J_l(\xi)},
\end{align*}
where we have used \eqref{appexp}.
Now since we are in resonant intervals for both, by definition of $w_R$, we have that 
\begin{align*}
\chi^{R,R}\frac{J_k(\eta)}{J_l(\xi)}&\approx\chi^{R,R} \frac{\frac{k^2}{\eta}(1+|t-\eta/k|)}{\frac{l^2}{\xi}(1+|t-\xi/l|)}\frac{w_{NR,k}(t,\eta)}{w_{NR,l}(t,\xi)}\\
&\lesssim\chi^{R,R}\frac{\frac{k^2}{\eta}(1+|t-\eta/k|)}{\frac{l^2}{\xi}(1+|t-\xi/l|)} e^{10\mu|k-l,\eta-\xi|^{1/2}}\\
&\lesssim \chi^{R,R}\frac{\xi}{l^2(1+|t-\xi/l|)}e^{10\mu|k-l,\eta-\xi|^{1/2}}\\
&=\chi^{R,R}\frac{w_{NR}(t,\xi)}{w_R(t,\xi)}e^{10\mu|k-l,\eta-\xi|^{1/2}}.
\end{align*} 
where in the last inequality we have used the fact that in resonant interval $I_{k,\eta}$, it holds that $|t-\eta/k|\lesssim \eta/k^2$. Now we have to use the trichotomy lemma \ref{trichotomy}. In particular, if \textit{(b)} holds, i.e. $|t-\xi/l|\gtrsim \xi/l^2$, then 
\begin{equation}
\label{RRwNR/wR}
\chi^{R,R}\frac{w_{NR}(t,\xi)}{w_R(t,\xi)}\lesssim 1
\end{equation}
If instead \textit{(c)} holds, then we have that $\xi/l\lesssim |\eta-\xi|$ (in the lemma we can exchange the role of $\eta$ and $\xi$ without any problem). And also in this case we have \eqref{RRwNR/wR}. \\
So we obtain that 
\begin{equation}
\label{chiRRJ/J}
\chi^{R,R}\frac{J_k(\eta)}{J_l(\xi)}\lesssim\chi^{R,R} e^{10\mu|k-l,\eta-\xi|^{1/2}}.
\end{equation}
Now, we cannot hope to gain integrability by the $\Delta_L^{-1}$, since in resonant interval the time does not play any role. But thanks to \eqref{RRwNR/wR}, we can essentially multiply by one to absorb the term $R_N^{R,R}$ in the $CK_w$. In fact just rewrite \eqref{chiRRJ/J} as follows 
\begin{align*}
\chi^{R,R}\frac{J_k(\eta)}{J_l(\xi)}&\lesssim \chi^{R,R} \frac{w_{NR}(t,\xi)}{w_{R}(t,\xi)}\frac{w_{R}(t,\xi)}{w_{NR}(t,\xi)}e^{10\mu|k-l,\eta-\xi|^{1/2}}\\
&\lesssim \chi^{R,R}\frac{w_{R}(t,\xi)}{w_{NR}(t,\xi)}\sqrt{\frac{\dt w_k(t,\eta)}{w_k(t,\eta)}}\sqrt{\frac{ w_l(t,\xi)}{\dt w_l(t,\xi)}} e^{11\mu|k-l,\eta-\xi|^{1/2}},
\end{align*}
where we have used \eqref{RRwNR/wR} and \eqref{dtw/wbuono}, absorbing Sobolev regularity in the exponential, see \eqref{appsobexp}. Notice also that we have divided by $\dt w_l(\xi)$, but since we are in resonant interval we know that $\dt w_l(\xi)\neq0$.\\
So putting together previous estimates, we infer that
\begin{align*}
\chi^{R,R}A(|\cdot|\circledast|\cdot|) &\lesssim \chi^{R,R}A_l(\xi)e^{c\lambda|k-l,\eta-\xi|^s}|l,\xi||k-l,\eta-\xi|\frac{J_k(\eta)}{J_l(\xi)}\\
&\lesssim \chi^{R,R}\sqrt{\frac{\dt w}{w}}\bigg(\sqrt{\frac{w}{\dt w}}\frac{w_R}{w_{NR}}|\cdot|A\circledast |\cdot| e^{\lambda|\cdot|^s}\bigg).
\end{align*}
Then, we are ready to bound $R_N^{R,R}$, in fact we have that 
\begin{align*}
|R_N^{R,R}|&=\big|\scalar{\chi^{R,R}A\hf}{A\big(|\cdot|^\perp \hphi_N*|\cdot|\hfNe\big)}\big|\\
&\lesssim \bigg \langle\chi^{R,R}\sqrt{\frac{\dt w}{w}} A|\hf|,\bigg(\sqrt{\frac{ w}{\dt w}}\frac{w_R}{w_{NR}}|\cdot|A|\hphi|_N*|\cdot|e^{\lambda|\cdot|^s}|\hf|_{<N/8}\bigg) \bigg \rangle\\
&\lesssim \Normabigg{\sqrt{\frac{\dt w}{w}}\tilde{A} f_{\sim N}}{L^2}\Normabigg{\mathds{1}_{t\in I_{l,\xi}}\sqrt{\frac{w}{\dt w}}\frac{w_R}{w_{NR}}|\nabla|\Delta_L^{-1}\tilde{A}f_N}{L^2}\Norma{f}{\mathcal{G^{\lambda,\sigma}}},
\end{align*}
where we have used \eqref{CS+Young} and the fact that $|k|\lesssim |\eta|$ and $|l|\lesssim |\xi|$ to substitute $A$ with $\tilde{A}$. Now observe that the middle term of the last inequality, it is exactly the same that appear in $R_N^{NR,R}$, hence applying \eqref{replacedtww} and bootstrap hypothesis, we have that
\begin{equation}
\label{boundRNRR}
|R_N^{R,R}|\lesssim \epsilon\Normabigg{\sqrt{\frac{\dt w}{w}}\tilde{A} f_{\sim N}}{L^2}^2,
\end{equation} 
proving the bound on $R_N^{R,R}$, hence putting together \eqref{boundRNNRNR}, \eqref{boundRNNRR}, \eqref{boundRNRNR} and \eqref{boundRNRR}, we have proved that 
\begin{equation}
\label{boundRN1}
|R_N^1|\lesssim \epsilon\Normabigg{\sqrt{\frac{\dt w}{w}}\tilde{A} f_{\sim N}}{L^2}^2 +\frac{\epsilon}{\langle t \rangle^{2s}}\Norma{|\nabla|^{s/2}Af_{\sim N}}{L^2}^2.
\end{equation}
\qed \\
Finally, putting together \eqref{boundRN3} and \eqref{boundRN1}, with standard Littlewood-Paley decomposition properties, see \ref{appelitt} we have proved proposition \ref{propreaction}. \qed
\section{Remainder term}
For the remainder term, the commutator does not help, so we treat the two terms separately. So we have to deal with 
\begin{align*}
\mathcal{R}_{N,N'}&=\scalar{Af}{A\big(u_N\cdot \nabla f_{N'}\big)}-\scalar{Af}{u_N\cdot \nabla Af_{N'}}\\
&:=\mathcal{R}_{N,N'}^1+\mathcal{R}_{N,N'}^2.
\end{align*}
On the support of the integrand we have that $|k-l,\eta-\xi|\approx|l,\xi|$. This implies also that $|k,\eta|\lesssim |l,\xi|$. In addition, for some $c\in(0,1)$, see \eqref{imptri}, it holds that
\begin{equation*}
|k,\eta|^s\leq c|k-l,\eta-\xi|^s+c|l,\xi|^s.
\end{equation*}
So let us start with the first term. Here it is enough to see how the multiplier $A$ distribute among $u_N$ and $f_{N'}$, in particular on the support of the integrand we have that 
\begin{align*}
A_k(\eta)&=e^{\lambda|k,\eta|^s}\langle k,\eta \rangle^{\sigma} J_k(\eta)\\
&\lesssim e^{c\lambda|k-l,\eta-\xi|^s}e^{c\lambda|l,\xi|^s}\langle k-l,\eta -\xi \rangle^{\sigma/2-1}\langle l,\xi \rangle^{\sigma/2+1}e^{20\mu|k,\eta|^{1/2}} \\
&\lesssim \langle \cdot \rangle^{\sigma/2+1}e^{\lambda |\cdot|^s}\circledast \langle \cdot \rangle^{\sigma/2-1}e^{\lambda|\cdot|^s},
\end{align*}
where in the last line we have used \eqref{appexp} to absorb exponential terms with $\mu$. Then we have directly that 
\begin{align*}
|\mathcal{R}_{N,N'}^1|&=\big|\scalar{A\hat{f}}{A\big(\hat{u}_N* |\cdot|\hat{f}_{N'} \big)}\big| \\
&\lesssim \scalar{A|\hat{f}|}{\big(\langle\cdot \rangle^{\sigma/2+1}e^{\lambda|\cdot|^s}|\hat{u}_{N}|*\langle\cdot \rangle^{\sigma/2}e^{\lambda|\cdot|^s}|\hat{f}_{N'}|\big)}\\
&\lesssim \norma{Af_{\sim_N}}{L^2}\norma{u_{\sim_N}}{\mathcal{G^{\lambda,\sigma-\text{3}}}}\norma{f_{\sim N}}{\mathcal{G^{\lambda,\sigma}}},
\end{align*}
where we have used \eqref{CS+Young}. Hence, thanks to bootstrap hypothesis and lemma \ref{ellipticlemma}, we infer that 
\begin{equation*}
|\mathcal{R}_{N,N'}^1|\lesssim \frac{\epsilon^3}{\langle t \rangle^{2}}.
\end{equation*}
The same estimate, with essentially the same technique, is obtained for $\mathcal{R}_{N,N'}^2$, hence after summing up all frequencies, and thanks to standard Litllewood-Paley decomposition properties, we have that
\begin{equation}
|\mathcal{R}|\lesssim \frac{\epsilon^3}{\langle t \rangle^2},
\end{equation}
hence proving proposition \ref{propremainders}. \qed
\subsection*{Conclusions}
To summarize, the proof is essentially the same as \cite{BM}, with simplifications in the number of terms to control, the possibility of being not very careful in elliptic regularity and without taking care about the change of coordinates. Anyway the ideas and the mathematical techniques are the same of \cite{BM}, we have just simplified (very little) some estimates for some Fourier multiplier that can work also in the general case.\\ \indent
We stress again that this note should be considered useful as a first approach to those techniques, in order to simplify the reading of \cite{BM} and related works.
\subsection*{Acknowledgements}
The author acknowledge Jacob Bedrossian, Michele Coti Zelati, Paolo Antonelli and Pierangelo Marcati for discussion about the problem and feedbacks on this note.  
        \appendix 
        \section{}
        \subsection{Littlewood-Paley decomposition}\label{appelitt}
        Here we recall the basic properties of the Littlewood-Paley decomposition and paraproducts, since are necessary throughout all the proof. For more details see \cite{bahouri2011fourier}. \\
        Consider as Fourier variable $\xi \in \mathbb{Z}^n\times \mathbb{R}^n$. Let $\chi$ a smooth cut-off function such that $\chi(\xi)=1$ for $|\xi|\leq 1/2$, and $\chi(\xi)=0$ for $|\xi|\geq 3/4$. Then $\varphi(\xi):=\chi(\xi/2)-\chi(\xi)$ is a smooth cut-off supported on the annulus $\{1/2\leq |\xi| \leq 3/2 \}$. Then in general define $\varphi_N(\xi)= \varphi(N^{-1}\xi)$, with the following support
        \begin{equation*}
        \text{supp} (\varphi_N)=\{N/2\leq|\xi|\leq 3N/2 \}
        \end{equation*}
        for $N\in \mathbf{D}$, the dyadic integers. Then one has that 
        		\begin{equation*}
        		 \chi(\xi)+\lim_{N'\to \infty}\sum_{N=1}^{N'}\varphi_N(\xi)=\lim_{N'\to \infty} \chi(\xi/N')=1,
        		\end{equation*}
        		hence we have a partition of unity. So given a function $g\in L^2(\mathbb{T}^n\times \mathbb{R}^n)$, we define 
        		\begin{align*}
        		g_N&:=\mathcal{F}^{-1}\big(\varphi_N \hat{g}\big),\\
        		g_{\frac{1}{2}}&:=\mathcal{F}^{-1}\big(\chi \hat{g}\big),\\
        		g_{<N}&:=g_{\frac{1}{2}}+\sum_{N'\in \mathbf{D}: N'<N}g_{N'},
        		\end{align*}
        		which means that we cut the functions on the frequency space. By linearity of the Fourier transform and Plancharel then it holds 
        		\begin{equation*}
        		g=\sum_{N\in \mathbf{D}}g_N,
        		\end{equation*}
        		where clearly it is used also that we have a partition of unity. So this is not a projection, since the support of the cut-off intersect, but we have the almost projection property, namely 
        		\begin{align*}
        		\norma{g}{L^2}^2&\approx \sum_{N\in \mathbf{D}}\norma{g_N}{L^2}^2
        		\\ \norma{g_N}{L^2} &\approx \norma{(g_N)_N}{L^2}.
        		\end{align*}
        	Also it is useful to define the following 
        	\begin{equation*}
        	g_{\sim N}=\sum_{N'\in \mathbf{D}:\ cN\leq N'\leq CN}g_{N'},
        	\end{equation*}
        	and clearly one has that 
        	\begin{equation*}
        	\norma{g_N}{L^2}\leq \norma{g_{\sim N}}{L^2}.
        	\end{equation*}
        	The Littlewood-Paley decomposition is very useful also for the following property 
        	\begin{equation*}
        	\norma{\nabla g_N}{L^2}\approx N\norma{g_N}{L^2}.
        	\end{equation*}
        	Finally, the paraproduct decomposition is just a multiplication by $1$, thanks to the previous partition of unity, and some rearrangement we obtain that, for any functions $f,g \in L^2(\mathbb{T}^n\times \mathbb{R}^n)$, it holds
        	\begin{equation*}
        	fg=\sum_{N\leq 8}f_{N}g_{<N/8}+\sum_{N\leq 8} f_{<N/8}g_{N}+\sum_{N}\sum_{N/8\leq N'\leq 8N}f_Ng_{N'}.
        	\end{equation*}
        	\subsection{Useful inequalities}
        	Here we recall some basic inequality but that are very useful for the proof.
        	\begin{lemma}
        		\label{lemmaconc} Let $0<s<1$ and assume, without loss of generality, that $x, y\geq 0$. Then 
        		\begin{description}
        			\item[(i)] In general we have the triangle inequalities 
        			\begin{align}
        			\label{triineq} \langle x+y\rangle^s &\leq \langle x \rangle^s+\langle y \rangle^s,\\
        			|\langle x \rangle^s-\langle y \rangle^s| &\leq \langle x-y\rangle^s,\\ 
        			C_s\big(\langle x \rangle^s+\langle y \rangle^s\big)&\leq  \langle x+y\rangle^s,
        			\end{align}
        			for some $C_s>0$ depending only on $s$.
        			\item[(ii)] In general  
        			\begin{equation}
        			\label{concconc}
        			\big|\langle x\rangle ^s-\langle y\rangle^s\big|\lesssim_s \frac{1}{\langle x \rangle^{1-s}+\langle y \rangle^{1-s}}\langle x-y\rangle.
        			\end{equation}
        			\item[(iii)] If $|x-y|\leq x/C$ for some $C>1$ then we have the improved triangular inequality
        			\begin{equation}
        			\label{imprevtri}
        			|\langle x \rangle^s-\langle y \rangle^s| \leq \frac{s}{(C-1)^{1-s}}\langle x-y\rangle^s.
        			\end{equation}
        			\item[(iv)] If $x\geq y$ we have also another improved triangular inequality 
        			\begin{equation}
        			\label{imptri}
        			\langle x+y\rangle^s \leq \bigg(\frac{\langle x \rangle}{\langle x \rangle+\langle y \rangle}\bigg)^{1-s}\big(\langle x \rangle^s+\langle y \rangle^s\big)
        			\end{equation}
        		\end{description}
        	\end{lemma}
        \begin{proof}
        	For the proof the fact that we have the japanese brackets does not play a real role. Then the proof of point $(i)$ it is standard, actually if $s\in \mathbb{Q}$ just follows by Newton's binomial formula. Point $(ii)$ essentially is the mean value theorem. \\
        	Point $(iii)$ comes from concavity properties, and the fact that $y\geq (C-1)x/C$, in fact
        	\begin{equation*}
        	x^s\leq y^s +\frac{s}{y^{1-s}}(x-y)\leq y^s+\frac{Cs}{\big((C-1)x\big)^{1-s}}(x-y),
        	\end{equation*}
        	and using $|x-y|\leq x/C$ we prove \eqref{imprevtri}. \\
        	Point $(iv)$ just follows by observing that 
        	\begin{equation*}
        	|x+y|^s=\frac{|x+y|}{|x+y|^{1-s}}\leq \bigg(\frac{x}{|x+y|}\bigg)^{1-s}(x^s+y^s).
        	\end{equation*}
        \end{proof}
    To relate Gevrey and Sobolev regularity we have the following.
    \begin{lemma}
    	\label{lemmaexp} Let $x\geq0$, then:
    	\begin{description}
    		\item[(i)] Let $\alpha>\beta\geq0$ and $C,\delta>0$, 
    		\begin{equation}
    		\label{appexp} \exp\big(Cx^\beta\big)\leq \exp\bigg(C\big(\frac{C}{\delta}\big)^{\frac{\beta}{\alpha-\beta}}\bigg)\exp(\delta x^\alpha).
    		\end{equation}
    		\item[(ii)] Let $\alpha, \sigma, \delta>0$, then 
    		\begin{equation}
    		\label{appsobexp}
    		\langle x\rangle^\sigma\lesssim \delta^{\frac{\sigma}{\alpha}}\exp\big(\delta x^\alpha\big).
    		\end{equation}
    	\end{description}
    \end{lemma}
    Then we recall some inequalities to deal with convolutions.
        \begin{lemma}
        	Let $f(\cdot), g(\cdot)\in L^2_\xi(\mathbb{R}^d)$, $\langle \cdot \rangle^\sigma h(\cdot) \in L^2_\xi(\mathbb{R}^d)$ and $\langle \cdot \rangle^\sigma b(\cdot)\in L^2_\xi(\mathbb{R}^d)$, for $\sigma>d/2$. Then 
        	\begin{align}
        	\label{Young_inq}
        	\norma{f*h}{L^2}&\lesssim \norma{f}{L^2}\norma{\langle \cdot \rangle^\sigma h}{L^2}\\
        	\label{CS+Young}|\scalar{f}{g*h}|&\lesssim \norma{f}{L^2}\norma{g}{L^2}\norma{\langle \cdot \rangle^\sigma h}{L^2}\\
        	\label{CS+2Young}|\scalar{f}{g*h*b}|&\lesssim \norma{f}{L^2}\norma{g}{L^2}\norma{\langle \cdot \rangle^\sigma h}{L^2}\norma{\langle \cdot \rangle^\sigma b}{L^2}
        	\end{align}
        \end{lemma}
    \begin{proof}
    	Inequality \eqref{Young_inq} is just a standard Young's inequality followed by Cauchy-Schwarz, since 
    	\begin{equation*}
    	\norma{h}{L^1}\leq \norma{\langle \cdot \rangle^{-\sigma}}{L^2}\norma{\langle \cdot \rangle^\sigma h}{L^2}\lesssim \norma{\langle \cdot \rangle^\sigma h}{L^2},
    	\end{equation*}
    	and the last follows by the fact that $\sigma>d/2$. \\
    	The inequality \eqref{CS+Young} is Cauchy-Schwarz plus \eqref{Young_inq}. Finally \eqref{CS+2Young} is Cauchy-Schwarz followed by two applications of Young's inequality.
    \end{proof}

  \bibliographystyle{siam}
  \nocite{*}
\bibliography{bob}
	\end{document}